\documentclass[11pt,reqno]{amsart} 

\usepackage[english]{babel} 
\usepackage[ansinew]{inputenc} 
\usepackage{amsmath} 
\usepackage{amsfonts} 
\usepackage{amssymb} 
\usepackage{amsthm} 
\usepackage{enumerate,soul}
\usepackage{enumitem} 
\usepackage[a4paper]{geometry} 
\usepackage[pdftex]{graphicx} 

\usepackage{afterpage}

\usepackage{fancyhdr} 
\usepackage{ifthen}

\newcommand\restr[2]{{
\left.\kern-\nulldelimiterspace 
#1 
\right|_{#2} 
}}

\DeclareMathOperator*{\GS}{GS} 
\DeclareMathOperator*{\GA}{GA} 
\DeclareMathOperator*{\FGS}{FGS} 
\DeclareMathOperator*{\supp}{supp} 
\renewcommand\Im{Im} 
\renewcommand\Re{Re} 

\def\ds{\displaystyle} 
\def\R{\mathbb R} 
\def\C{\mathbb C} 
\def\N{\mathbb N}

\newcommand{\D}{\mathcal{D}} 
 
\newcommand{\F}{\mathcal{F}} 
 
\newcommand{\Sch}{\mathcal{S}}

\AtBeginDocument{
\def\MR#1{}} 

\let\OLDthebibliography=\thebibliography 
\def\thebibliography#1{ 
\OLDthebibliography{#1} 
\addcontentsline{toc}{section}{\refname}} 

\newtheorem{theo}{Theorem}[section]
\newtheorem{prop}[theo]{Proposition} 
\newtheorem{defin}[theo]{Definition} 
\newtheorem{lema}[theo]{Lemma} 
\newtheorem{cor}[theo]{Corollary} 
\newtheorem{nota}[theo]{Remark} 
\newtheorem{exam}[theo]{Example} 

\catcode`\@=11

\renewcommand{\section}%
   {\setcounter{equation}{0}\@startsection {section}{1}{\z@}{-3.5ex plus -1ex
  minus -.2ex}{2.3ex plus .2ex}{\Large\bf}}

\begin{document} 

\title[Global pseudodifferential operators of infinite order...]{Global pseudodifferential operators of infinite order in classes of ultradifferentiable functions}

\author[Asensio]{Vicente Asensio}
\address{
Instituto Universitario de Matem\'atica Pura y Aplicada IUMPA\\
Universitat Po\-li\-t\`ecni\-ca de Val\`encia\\
Camino de Vera, s/n\\
E-46071 Valencia\\
Spain}
\email{viaslo@upv.es}

\author[Jornet]{David Jornet}
\address{
Instituto Universitario de Matem\'atica Pura y Aplicada IUMPA\\
Universitat Po\-li\-t\`ecni\-ca de Val\`encia\\
Camino de Vera, s/n\\
E-46071 Valencia\\
Spain}
\email{djornet@mat.upv.es}


\markboth{\sc Global pseudodifferential operators...}
 {\sc  Vicente Asensio and David Jornet}

\begin{abstract}
We develop a theory of pseudodifferential operators of infinite order for the global classes $\Sch_{\omega}$ of ultradifferentiable functions in the sense of Bj\"orck, following the previous ideas given by Prangoski for ultradifferentiable classes in the sense of Komatsu. We study the composition and the transpose of such operators with symbolic calculus and provide several examples.
\end{abstract}

\maketitle

\section{Introduction}

The \emph{local} theory of pseudodifferential operators grew out of the study of singular integral operators, and developed after 1965 with the systematic studies of Kohn-Nirenberg~\cite{KN}, H\"ormander~\cite{H}, and others. Since then, several authors have studied pseudodifferential operators of finite or infinite order in Gevrey classes in the local sense; we mention, for instance, \cite{HMM,Z}.  We refer to Rodino~\cite{R} for an excellent introduction to this topic, and the references therein.

Gevrey classes are spaces of (non-quasianalytic) ultradifferentiable functions in between real analytic  and $C^{\infty}$ functions. The study of several problems in general classes of ultradifferentiable functions has received much attention in the last 60 years. Here, we will work with ultradifferentiable functions as defined by Braun, Meise and Taylor~\cite{BMT}, which define the classes in terms of the growth of the derivatives of the functions,  or in terms of the growth of their Fourier transforms (see, for example, Komatsu~\cite{K} and Bj\"orck~\cite{bjorck1966linear}, or \cite{BMT}, for two different points of view to define spaces of ultradifferentiable functions and ultradistributions; and \cite{BMM} for a comparison between  the classes defined in \cite{BMT} and \cite{K}). 

In \cite{FGJ2005pseudo}, a full theory of pseudodifferential operators in the local sense is developed for ultradifferentiable classes of Beurling type as in \cite{BMT}, and it is proved that the corresponding operators are $\omega$-pseudo-local, and the product of two operators is given in terms of a suitable symbolic calculus. In \cite{FGJ-hypo,FGJ-wfs} the same authors construct a parametrix for such operators and study the action of the wave front set on them (see also \cite{AJO-wfs} for a different point of view). On the other hand, very recently, Prangoski~\cite{Pran2013pseudo} studies pseudodifferential operators of {\em global} type and infinite order for ultradifferentiable classes of Beurling and Roumieu type in the sense of Komatsu, and later, in \cite{CPP},  a parametrix  is constructed for such operators. See \cite{Pran2013pseudo,Nic_Rod2010global} and the references therein for more examples of pseudodifferential operators in global classes (e.g., in Gelfand-Shilov classes).

Our aim is to study pseudodifferential operators of global type and infinite order in classes of ultradifferentiable functions of Beurling type as introduced in \cite{BMT}. Hence, the right setting is the class $\Sch_{\omega}$ as introduced by Bj\"orck~\cite{bjorck1966linear}. We follow the lines of Prangoski~\cite{Pran2013pseudo} and Shubin~\cite{Shu2001pseudo}, but from the point of view of \cite{FGJ2005pseudo}, in such a way that our proofs simplify the ones of \cite{Pran2013pseudo}. Moreover, we clarify the role of some kind of entire functions \cite{Braun1993an,Lan1996conti} that become crucial throughout the text. 

The paper is organized as follows. First, in Section~\ref{SectionBackground}, we introduce our setting, we give some useful results about the class $\Sch_{\omega}$ and we recall from \cite{Braun1993an,Lan1996conti} the existence of some kind of $\omega$-ultradifferential operators very useful in the next sections. In Section~\ref{section-pseudos} we introduce our symbol (amplitude) classes and define the corresponding pseudodifferential operators. We give in Proposition~\ref{PropExampleExtensionPseudodifferentialOperator} a characterization in terms of the kernel of an $\omega$-regularizing (pseudodifferential) operator, which are very important in the construction of parametrices of hypoelliptic operators. We see in Example~\ref{examples} that many operators are pseudodifferential operators according to our definition. In particular, we show that our classes of symbols are different from the ones of \cite{Pran2013pseudo}.  In Section~\ref{section-calculus} we develop the symbolic calculus and we state some previous results needed to compose two pseudodifferential operators. In Section~\ref{section-composition}, we study the composition of two of our operators. To this aim, we analyse carefully the behaviour of the kernel of a pseudodifferential operator outside the diagonal in Theorem~\ref{TheoFormalKernel}. This result is an improvement of \cite[Theorem 6.3.3]{Nic_Rod2010global} and \cite[Proposition 5]{Pran2013pseudo}. The results that we obtain let the study of parametrices for hypoelliptic differential operators in this setting.

 
\section{Preliminaries}\label{SectionBackground} 

We begin with some notation on multi-indices. Throughout the text we will denote by $\alpha = (\alpha_1,\ldots,\alpha_d) \in \mathbb{N}_{0}^{d}$ a multi-index of dimension $d$. We denote the length of $\alpha$ by \[ |\alpha| = \alpha_1 + \ldots + \alpha_d. \] For two multi-indices $\alpha$ and $\beta$ we write $\beta \leq \alpha$ for $\beta_j \leq \alpha_j,$ when $j=1,\ldots,d$. Moreover, $\alpha! = \alpha_1! \cdots \alpha_d!$ and if $\beta \leq \alpha$, then 
\[ \binom{\alpha}{\beta} := \binom{\alpha_1}{\beta_1} \cdots \binom{\alpha_d}{\beta_d} = \frac{\alpha!}{\beta!(\alpha-\beta)!}. \] We also write 
\[ \partial^{\alpha} = \big(\frac{\partial}{\partial x_1}\big)^{\alpha_1} \cdots \big(\frac{\partial}{\partial x_d}\big)^{\alpha_d},\] and using the notation $D_{x_j} = -i\frac{\partial}{\partial x_j}$, $j=1,\ldots,d$, where $i$ is the imaginary unit, we set 
\[ D^{\alpha} = D^{\alpha_1}_{x_1} \cdots D^{\alpha_d}_{x_d}. \] For $x=(x_{1},\ldots,x_{d}) \in \mathbb{R}^{d}$, let \[ x^{\alpha} = x_1^{\alpha_1} \cdots x_d^{\alpha_d}. \] 
We denote $\langle x \rangle = (1+|x|^2)^{\frac1{2}}$ for every $x \in \mathbb{R}^{d}$, where $|x|$ is the Euclidean norm of $x$.  Our setting requires weight functions as defined by Braun, Meise and Taylor~\cite{BMT}. 

\begin{defin}\label{DefinitionWeightFunction} 
A non-quasianalytic weight function $\omega:[0,+\infty[ 
\to[0,+\infty[$ is a continuous and increasing function which satisfies: 
\begin{itemize} 
\item[($\alpha$)] 
\vspace{0.5mm} $\displaystyle\exists\ L\geq 1\ \mbox{s.t.}\ 
\omega(2t)\leq L(\omega(t)+1),\ \forall t\geq0,$ 
\item[($\beta$)] 
\vspace{0.5mm} $\displaystyle \int_{1}^{+\infty} 
\frac{\omega(t)}{t^2} dt < +\infty$, 
\item[($\gamma$)] 
\vspace{0.5mm} $\displaystyle \log(t) = o(\omega(t)) $ as $t \to 
\infty$,
\item[($\delta$)] 
\vspace{0.5mm} $\displaystyle \varphi: t \mapsto 
\omega(e^{t})\ $ is convex. 
\end{itemize} 
\end{defin}
Throughout the text, if necessary, we will denote $\varphi$ by $\varphi_{\omega}$ in some cases.

\begin{exam}{\rm The following functions are, after a
change
in some interval $[0,M]$,  examples of weight functions:
\begin{enumerate}
 \item[(i)] $\omega (t)=t^d$ for $0<d<1.$  
 
  \item[(ii)] $\omega (t)=\left(\log(1+t)\right)^s$, $s>1.$ 
  
   
    \item[(iii)] $\omega(t)=t^{d} (\log(e+t))^{s}$, $0<d<1, s\neq0.$
    \end{enumerate}
    }
\end{exam}

By definition, we extend the weight function in a radial way to $\mathbb{C}^{d}$, i.e. 
\[ \omega(\xi) = \omega(|\xi|), \ \ \xi=(\xi_1,\ldots,\xi_d) \in \mathbb{C}^{d}.\] 
We observe that there exists $L'>0$ depending on the constant $L>0$ of Definition~\ref{DefinitionWeightFunction}\,$(\alpha)$  and the dimension $d$ such that for any $x=(x_1,\ldots,x_d) \in \mathbb{C}^{d}$: 
\begin{eqnarray} \label{EqOmegax}
\omega(x) \le L' \omega(|x|_{\infty}) + L'  
\leq L'\omega(x) + L',
\end{eqnarray} where $|x|_{\infty}:=\max(|x_{1}|,\ldots,|x_{d}|)$. Moreover, as in \cite[Lemma 1.2]{BMT}, if $x,y \in \mathbb{C}^d$, then 
\begin{eqnarray} \label{EqOmegaDiff} 
\omega(x+y)  \leq L\big(\omega(x)+\omega(y)+1\big). 
\end{eqnarray} 

We will assume without loss of generality that $\omega|_{[0,1]}\equiv 0$, which gives some useful properties (see~\cite{BMT}). For instance, we have 
\begin{equation}\label{EqOmegaCorchetex} 
\omega(\langle x \rangle) = \omega\big(\sqrt{1+|x|^2}\big) \leq \omega(1+|x|) \leq L(\omega(x) + 1). 
\end{equation} 

We consider now property $(\delta)$ of Definition~\ref{DefinitionWeightFunction} and define: 
\begin{defin} 
The Young conjugate $\varphi^{\ast}:[0,\infty[ \to [0,\infty[$ of $\varphi$ is given by 
\[ \varphi^{\ast}(t) := \sup_{s \geq 0}\{st - \varphi(s) \}. \] 
\end{defin} 
Since $\omega|_{[0,1]}\equiv 0$, we have $\varphi^{\ast}(0)=0$. Moreover, $\varphi^{\ast}$ is convex, the function $\varphi^{\ast}(t)/t$ is increasing and $\varphi^{\ast \ast} = \varphi$. 

It is not difficult to prove the next two results; see, for instance, \cite[Lemma 1.4, Remark 1.7]{FGJ2005pseudo}. 
\begin{lema}\label{LemmaTechnical} 
For each $n,k \in \mathbb{N}$ and $t \geq 1$, we have 
\begin{eqnarray}\label{EqLemmaTechnical1} 
t^{k} &\leq& e^{n \varphi^{\ast}(\frac{k}{n})} e^{n\omega(t)}; \\ \label{EqLemmaTechnical2} 
\inf_{j \in \mathbb{N}_{0}} t^{-j} e^{k \varphi^{\ast}(\frac{j}{k})} &\leq& e^{-k\omega(t) + \log(t)}. 
\end{eqnarray} 
\end{lema} 
 
\begin{prop}\label{PropTechnical} 
If a weight function $\omega$ satisfies $\omega(t)=o(t^{a})$ as $t\to+\infty$ for some constant $0<a\le 1$, for every $B>0$ and $\lambda >0$, there exists a constant $C>0$ such that 
\begin{equation*}\label{EqPropTechnical} 
B^{n} n! \leq C e^{a\lambda \varphi^{\ast}(\frac{n}{\lambda})}, \qquad  n \in \mathbb{N}_{0}. 
\end{equation*} 
\end{prop} 
It is an exercise to see that:
\begin{lema}\label{LemmaCorcheteIneq} 
For every $(x,y) \in \mathbb{R}^{2d}$ we have 
\[ \langle x-y \rangle \leq \sqrt{2} \langle (x,y) \rangle.\] 
\end{lema} 


From the convexity of $\varphi^{*}$ and the fact that $\varphi^*(0)=0$ we have (see, for instance, \cite[Lemma 1.3]{FGJ2005pseudo}) 
\begin{lema}\label{lemmafistrella} 
\begin{enumerate} 
\item[(1)] Let $L>0$ be such that $\omega(e t)\le L(\omega(t)+1)$ (this is possible from Definition~\ref{DefinitionWeightFunction}($\alpha$)).  We have   
\begin{equation}\label{EqLLambdaVarphi} 
\lambda L^{n} \varphi^{\ast}\big(\frac{y}{\lambda L^{n}}\big) + ny \leq \lambda \varphi^{\ast}\big(\frac{y}{\lambda}\big) + \lambda \sum_{j=1}^{n} L^{j}, 
\end{equation} for every $y\geq 0$, $\lambda >0$ and $n \in \mathbb{N}$. 
\item[(2)] For all $s,t,\lambda>0$, we have 
\begin{equation*}
\label{EqEstimationVarphi} 
2\lambda \varphi^{\ast}\big(\frac{s+t}{2\lambda}\big) \leq \lambda\varphi^{\ast}\big(\frac{s}{\lambda}\big) + \lambda\varphi^{\ast}\big(\frac{t}{\lambda}\big) \leq \lambda\varphi^{\ast}\big(\frac{s+t}{\lambda}\big). 
\end{equation*} 
\end{enumerate} 

\end{lema} 

The following lemma is taken from~\cite[Lemma 1.5\,(2)]{FGJ2005pseudo}:
\begin{lema}\label{LemmaTechnical2} 
If $\frac{k}{N} \varphi^{\ast}\big(\frac{N}{k}\big) \leq \log(t) \leq \frac{k}{N+1}\varphi^{\ast}\big(\frac{N+1}{k}\big)$, then 
\[ t^{-N} e^{2k \varphi^{\ast}\big(\frac{N}{2k}\big)} \leq e^{-k\omega(t) + \log(t)}. \] 
\end{lema} 
It is not difficult to see the following
\begin{lema}\label{Lemma158TesisDavid}
Let $0<a\le 1$ be a constant and let $\omega$ and $\sigma$ be weight functions. Then:
\begin{enumerate}
\item[(1)] If $\omega\big(t^{\frac1{a}}\big) = o\big(\sigma(t)\big)$ as $t \to \infty$, for all $\lambda,\mu>0$ there exists $C:=C_{\lambda,\mu}>0$ such that
\[ \lambda \varphi_{\sigma}^{\ast}\big(\frac{j}{\lambda}\big) \leq C + \mu a \varphi_{\omega}^{\ast}\big(\frac{j}{\mu}\big), \quad j \in \mathbb{N}_0. \]

\item[(2)] If $\omega\big(t^{\frac1{a}}\big) = O\big(\sigma(t)\big)$ as $t \to \infty$,  there is $C>0$ such that for each $\lambda>0$,
\[ \lambda \varphi_{\sigma}^{\ast}\big(\frac{j}{\lambda}\big) \leq \lambda + a\lambda  C \varphi_{\omega}^{\ast}\big(\frac{j}{\lambda C}\big), \quad j\in\N_{0}. \]
\end{enumerate}

\end{lema}


We consider also  the {\em Fourier transform} of $u
\in L^1(\R^d)$ denoted by
$$
\F(u)(\xi)=\hat{u}(\xi):=\int_{\R^d}u(x)e^{-i\langle x,\xi\rangle}dx,\qquad\xi\in\R^d,
$$
with standard extensions to more general spaces of functions and distributions. We will work in the global spaces of ultradifferentiable functions and ultradistributions as defined by Bj\"orck~\cite{bjorck1966linear}: 
\begin{defin} 
\label{def3} 
For a weight $\omega$ as in Definition \ref{DefinitionWeightFunction} we define 
$\Sch_\omega(\R^{d})$ as 
the set of all $u\in L^1(\R^d)$ such that $u$ and its Fourier transform $\widehat{u}$ belong to $C^\infty(\R^d)$ and 
\begin{itemize} 
\item[(i)] for each $\lambda>0$ and $\alpha\in\N^d_0,\quad\ds 
\sup_{x\in\R^d}e^{\lambda\omega(x)}|D^\alpha u(x)|<+\infty,$ 
\item[(ii)] for each 
$\lambda>0$ and $\alpha\in\N^d_0,\quad\ds 
\sup_{\xi\in\R^d}e^{\lambda\omega(\xi)}|D^\alpha\widehat{u}(\xi)|<+\infty.$ 
\end{itemize} 
As usual, the corresponding dual space is denoted by 
$\Sch'_\omega(\R^{d})$ and is the set of all the linear and continuous 
functionals $u:\Sch_\omega(\R^{d})\to\C$. We say that an element of 
$\Sch'_\omega(\R^{d})$ is an {\em $\omega$-temperate ultradistribution.} 
\end{defin} 
Now, we give a useful characterization of $\Sch_\omega(\R^{d})$. See \cite{BJO2017Regu} for an exhaustive  characterization of the space $\Sch_\omega(\R^{d})$ in terms of seminorms. 

\begin{lema}\label{LemmaEquivalence} 
If $f \in \mathcal{S}(\R^{d})$, then $f \in \Sch_\omega(\R^{d})$ if and only if for every $\lambda,\mu>0$ there is $D_{\lambda,\mu}>0$ such that for all $\alpha \in \mathbb{N}_{0}^{d}$ and $x\in\R^d,$ we have 
\begin{equation}\label{EqDalpha} 
|D^{\alpha} f(x)| \leq D_{\lambda,\mu} e^{\lambda \varphi^{\ast}\big( \frac{|\alpha|}{\lambda} \big) } e^{-\mu \omega(x)}. 
\end{equation} 
\end{lema} 
\begin{proof} 
If $f \in \Sch_\omega(\R^{d})$, by \cite[Theorem 4.8]{BJO2017Regu} we have that for all $\lambda,\mu >0$ there exists $C_{\lambda,\mu}>0$ such that 
\begin{equation}\label{EqXbetaDalpha} 
\sup_{x \in \mathbb{R}^{d}} |x^{\beta} D^{\alpha} f(x)| \leq C_{\lambda,\mu} e^{\lambda \varphi^{\ast}\big( \frac{|\alpha|}{\lambda} \big)} e^{\mu \varphi^{\ast} \big(\frac{|\beta|}{\mu}\big)}, \qquad  \alpha,\beta \in \mathbb{N}_{0}^{d}. 
\end{equation} We fix $\beta = (\beta_1, \ldots, \beta_d) \in \mathbb{N}_{0}^{d}$ and $x=(x_1,\ldots,x_d) \in \mathbb{R}^{d}$. Assume w.l.o.g. that $|x_1|=|x|_{\infty}\ge 1$. We have
%
%
\[ 
|x^{\beta}D^{\alpha}f(x)| = |x_1|^{\beta_1} \ldots |x_d|^{\beta_d} |D^{\alpha}f(x)| \leq |x_1|^{\beta_1+\ldots+\beta_d} |D^{\alpha} f(x)| = |x^{\widetilde{\beta}} D^{\alpha}f(x)|, 
\] 
where $\widetilde{\beta} = \big( \beta_1 + \ldots + \beta_d, 0, \ldots, 0\big) \in \mathbb{N}_{0}^{d}$ and, obviously, $|\widetilde{\beta}|=|\beta|$. We apply our hypothesis~\eqref{EqXbetaDalpha} to $\alpha$ and $\widetilde{\beta}$ to obtain 
\begin{equation}\label{EqLemmaEquivalence} 
|x_1|^{\beta_1+\ldots+\beta_d} |D^{\alpha}f(x)| = |x^{\widetilde{\beta}} D^{\alpha}f(x)| \leq C_{\lambda,\mu} e^{\lambda \varphi^{\ast}\big(\frac{|\alpha|}{\lambda}\big)}e^{(\mu L'+1) \varphi^{\ast}\big(\frac{|\beta|}{\mu L'+1}\big)} 
\end{equation} 
for a positive constant $C_{\lambda,\mu}$ where $L'>0$ is the constant of \eqref{EqOmegax}. Now, we put $j:=\beta_1+\ldots+\beta_d =|\beta|\in \mathbb{N}_{0}$ in formula~\eqref{EqLemmaEquivalence} to obtain, by~\eqref{EqLemmaTechnical2} and \eqref{EqOmegax}, 
\begin{eqnarray} 
\nonumber |D^{\alpha} f(x)| &\leq& C_{\lambda,\mu} e^{\lambda \varphi^{\ast}\big(\frac{|\alpha|}{\lambda}\big)} \Big( \inf_{j \in \mathbb{N}_{0}} |x_1|^{-j} e^{(\mu L'+1) \varphi^{\ast}\big(\frac{j}{\mu L'+1}\big)}\Big)\\ \label{EqLemmaIneqDalpha} 
&\leq&\nonumber C_{\lambda,\mu} e^{\lambda \varphi^{\ast}\big(\frac{|\alpha|}{\lambda}\big)} e^{-(\mu L'+1) \omega(|x_1|) + \log|x_1|}\\&\le& C'_{\lambda,\mu} e^{\lambda \varphi^{\ast}\big(\frac{|\alpha|}{\lambda}\big)} e^{-\mu L' \omega(|x_1|)}\le e^{\mu L'}C'_{\lambda,\mu} e^{\lambda \varphi^{\ast}\big(\frac{|\alpha|}{\lambda}\big)} e^{-\mu \omega(x)},
\end{eqnarray} 
for some new constant $C'_{\lambda,\mu}>0$. 

Conversely, by~\eqref{EqLemmaTechnical1}, for $|x|\ge 1$ and any $\mu>0,$ we have $|x^{\beta}| \leq |x|^{|\beta|}  \leq e^{\mu \varphi^{\ast}\big(\frac{|\beta|}{\mu}\big)} e^{\mu \omega(x)}.$ Thus, by our hypothesis~\eqref{EqDalpha}, for each $\alpha, \beta\in\N_{0}^{d}$ and $x\in\R^{d}$, we get
\begin{eqnarray*} 
|x^{\beta} D^{\alpha}f(x)| \le |x|^{|\beta|} |D^{\alpha}f(x)| \leq D_{\lambda,\mu} e^{\lambda \varphi^{\ast}\big(\frac{|\alpha|}{\lambda}\big)} e^{\mu \varphi^{\ast}\big(\frac{|\beta|}{\mu}\big)}, 
\end{eqnarray*} 
which concludes the proof. 
\end{proof} 

\begin{nota}\label{Notaflambda} 
For $\lambda>0$, we denote for $f \in \Sch_\omega(\R^{d})$, \[ |f|_{\lambda} := \sup_{\alpha \in \mathbb{N}_{0}^{d}} \sup_{x \in \mathbb{R}^{d}}  |D^{\alpha} f(x)| e^{-\lambda \varphi^{\ast} \big(\frac{|\alpha|}{\lambda} \big)} e^{\lambda \omega(x)}\ , \] which is a seminorm. Observe that for any $x \in \mathbb{R}^{d}, \lambda>0$ and $\alpha \in \mathbb{N}_{0}^{d}$, we have 
\begin{eqnarray*} 
|D^{\alpha}f(x)| &=& |D^{\alpha} f(x)| e^{-\lambda \varphi^{\ast} \big(\frac{|\alpha|}{\lambda} \big)} e^{\lambda \omega(x)} e^{\lambda \varphi^{\ast} \big(\frac{|\alpha|}{\lambda} \big)} e^{-\lambda \omega(x)} \\ 
&\leq& |f|_{\lambda} e^{\lambda \varphi^{\ast} \big(\frac{|\alpha|}{\lambda} \big)} e^{-\lambda \omega(x)}. 
\end{eqnarray*} 
By Lemma~\ref{LemmaEquivalence}, $\big\{| \cdot |_{\lambda}\big\}_{\lambda > 0}$ is a fundamental system of seminorms in the class $\Sch_\omega(\R^{d})$. 
\end{nota} 

We write $P(\xi,r)$ for the polydisc of center $\xi = (\xi_1,\ldots,\xi_d) \in \mathbb{C}^d$ and polyradius $r=(r_1,\ldots,r_d)$, where each $r_j$ is positive, $j=1,\ldots,d$. That is, 
\[ P(\xi,r) = D(\xi_1,r_1) \times \ldots \times D(\xi_d,r_d) = \big\{ \zeta=(\zeta_1,\ldots,\zeta_d) \in \mathbb{C}^{d}: |\zeta_j - \xi_j| < r_j, \quad 1\leq j \leq d \big\}.\] 
And, also, 
\[ \partial P(\xi,r) := \{z \in \mathbb{C}^{d}:\, |z_j-\xi_{j}|=r_j,\,j=1,\ldots,d \}. \] 

Let us recall the following results on several complex variables. 

\begin{theo}[Cauchy's integral formula for the derivatives]\label{TheoCauchyIntegralFormula} 
Let $\Omega \subset \mathbb{C}^d$ be an open set, $a \in \Omega$ and $r=(r_1,\ldots,r_d) \in \mathbb{R}^{d}, r_j>0$ for every $j=1,\ldots,d$ so that $\overline{P(a,r)} \subset \Omega$. Let $f:\Omega \to \mathbb{C}$ be continuous and partially holomorphic. Then for all $\alpha \in \mathbb{N}_{0}^{d}$ and all $z \in P(a,r)$: 
\[ D^{\alpha}f(z) = \frac{\alpha!}{(2\pi i)^{d}} \int_{\partial P(a,r)} \frac{f(\xi)}{(\xi-z)^{\alpha+(1,\ldots,1)}} d\xi. \] 
\end{theo} 

\begin{prop}[Cauchy's inequalities]\label{LemmaCauchyIneq} 
Under the assumptions of Theorem~\ref{TheoCauchyIntegralFormula}, for every multi-index $\beta \in \mathbb{N}_{0}^{d}$, the following formula holds: 
\[ \big| D^{\beta}f(a) \big| \leq \sup_{\zeta \in \partial P(a,r)} \{ |f(\zeta)| \} \frac{\beta!}{r^{\beta}}. \] 
\end{prop} 

Now, we need to introduce the following space of functions (see \cite{BMT}, \cite{FGJ2005pseudo}). Let $\omega$ be a weight function. For an open set $\Omega \subset \mathbb{R}^d$, we define the space of {\em ultradifferentiable functions of Beurling type} in $\Omega$ as
\[ \mathcal{E}_{(\omega)}(\Omega) := \big\{ f \in C^{\infty}(\Omega): |f|_{K,\lambda}<\infty \ \text{for every} \ \lambda>0, \ \text{and every} \ K \subset \Omega \ \text{compact} \big\}, \] 
where 
\[ |f|_{K,\lambda}:=  \sup_{\alpha \in \mathbb{N}_0^d} \sup_{x \in K}|D^{\alpha}f(x)| e^{-\lambda \varphi^{\ast}\big(\frac{|\alpha|}{\lambda}\big)}. \] 
We endow such space with the Fr\'echet topology given by the sequence of seminorms $|f|_{K_n,n}$, where $(K_n)_{n}$ is any compact exhaustion of $\Omega$ and $n\in\N$. The strong dual of $\mathcal{E}_{(\omega)}(\Omega)$ is the space of compactly supported ultradistributions of Beurling type and is denoted by $\mathcal{E}'_{(\omega)}(\Omega)$. 

The space of {\em ultradifferentiable functions of Beurling type with compact support} in $\Omega$ is denoted by $\D_{(\omega)}(\Omega),$ and it is the space of those functions  $f\in \mathcal{E}_{(\omega)}(\Omega)$ such that its support, denoted by $\supp f$, is compact in $\Omega$. Its corresponding dual space is denoted by $\D'_{(\omega)}(\Omega)$ and it is called the space of {\em ultradistributions of Beurling type} in $\Omega$.

We also need the notion of $(\omega)$-\emph{ultradifferential operator} with constant coefficients, which  plays an important role in structure theorems for ultradistributions \cite{Braun1993an,K}.  Let $G$ be an entire function in $\mathbb{C}^d$ with $\log|G| = O(\omega)$. For $\varphi \in \mathcal{E}_{(\omega)}\big(\mathbb{R}^{d}\big)$, the map $T_{G}:\mathcal{E}_{(\omega)}\big(\mathbb{R}^{d}\big)\to\C$ given by
\[ T_{G}(\varphi) := \sum_{\alpha \in \mathbb{N}_{0}^{d}}  \frac{D^{\alpha}G(0)}{\alpha!} D^{\alpha}\varphi(0) \] 
defines an ultradistribution $T_{G} \in \mathcal{E}'_{(\omega)}\big(\mathbb{R}^{d}\big)$ with support equal to $\{0\}$. The convolution operator $G(D): \mathcal{D}'_{(\omega)}\big(\mathbb{R}^{d}\big) \to \mathcal{D}'_{(\omega)}\big(\mathbb{R}^{d}\big)$ defined by $G(D)(\mu) = T_{G} \ast \mu$ is said to be an ultradifferential operator of ($\omega$)-class. 

The following result is due to Langenbruch~\cite[Corollary 1.4]{Lan1996conti}. It shows the existence of entire functions with prescribed exponential growth (cf. \cite[Theorem 7]{Braun1993an}). 
\begin{theo}\label{TheoLangenbruch} 
Let $\omega:[0,\infty[ \to [0,\infty[$ be a continuous and increasing function satisfying the conditions $(\alpha)$,$(\gamma)$ and $(\delta)$ of Definition~\ref{DefinitionWeightFunction}. Then there exist an even function $f \in \mathcal{H}(\mathbb{C})$ and  $C_1$, $C_2$, $C_3>0$ such that 
\begin{enumerate} 
\item[i)] $\displaystyle \log{|f(z)|} \leq \omega(z) + C_1, \quad z \in \mathbb{C}$; 
\item[ii)] $\displaystyle \log{|f(z)|} \geq C_2\omega(z), \quad \text{for} \ \ z \in U:=\big\{ z \in \mathbb{C}: |\Im(z)| \leq C_3(|\Re(z)| + 1) \big\}$. 
\end{enumerate} 
\end{theo} 

From this result we deduce the analogous statement for several variables. 

\begin{theo}\label{TheoLangenbruchSeveralVariables} 
Let $\omega$ satisfy the hypotheses of Theorem~\ref{TheoLangenbruch}. Then there are a function $G \in \mathcal{H}\big(\mathbb{C}^{d}\big)$ and some constants $C_1,C_2,C_3,C_4>0$ such that 
\begin{enumerate} 
\item[i')] $\displaystyle \log{|G(z)|} \leq \omega(z) + C_1, \quad z \in \mathbb{C}^{d}$; 
\item[ii')] $\displaystyle \log{|G(z)|} \geq C_2\omega(z) - C_4, \quad \text{for} \ \ z \in \widetilde{U}:=\big\{ z \in \mathbb{C}^{d}: |\Im(z)| \leq C_3(|\Re(z)| + 1) \big\}$. 
\end{enumerate} 
\end{theo} 

\begin{proof} 
By Theorem~\ref{TheoLangenbruch}, there exist an even function $f \in \mathcal{H}(\mathbb{C})$ and strictly positive constants $C_1,\widetilde{C_2},\widetilde{C_3}$ such that 
\begin{eqnarray}\label{EqTheoLangenbruchFz1} 
\log{|f(z)|} &\leq& \omega(z) + C_1, \quad  z \in \mathbb{C}; \\ \label{EqTheoLangenbruchFz2} 
\log{|f(z)|} &\geq& \widetilde{C_2}\omega(z), \quad \text{for} \ \ z \in U:=\big\{ z \in \mathbb{C}: |\Im(z)| \leq \widetilde{C_3}(|\Re(z)| + 1) \big\}. 
\end{eqnarray} 
Since $f$ is even, \[ f(z) = \sum_{n=0}^{\infty} a_n z^{2n} \] for some $\{ a_n \}_{n=0}^{\infty} \subseteq \mathbb{C}$. We observe that $\log|f(0)| \geq 0$ by formula~\eqref{EqTheoLangenbruchFz2}, and then, $a_0$ is not zero. Now, for a fixed $z=(z_1,\ldots,z_d) \in \mathbb{C}^{d} \setminus \{ 0 \}$, we set $w=\sqrt{z_1^2 + \ldots + z_d^2}$ (here we consider a square root for which $w$ is well defined) and define \[ G(z) = \sum_{n=0}^{\infty} a_n(z_1^2 + \ldots + z_d^2)^n=f(w). \] $G$ is well defined and entire, according to the properties of $f$. Observe that, since $w=\sqrt{z_1^2 + \ldots + z_d^2} \in \mathbb{C}$, we have, by~\eqref{EqTheoLangenbruchFz1}, \[ \log|G(z)|=\log|f(w)| \leq \omega(w) + C_1. \] This proves condition $i')$, since 
\begin{eqnarray*} 
\omega(w) = \omega(|w|) 
\leq \omega\Big(\sqrt{|z_1^2| + \ldots + |z_d^2|}\Big) = \omega(z). 
\end{eqnarray*} 
On the other hand, to prove $ii')$, first we observe that for a small enough $0<\varepsilon<1$, $|\Im(z)| < \varepsilon|\Re(z)|$ implies that $w \in U$. 
%
Therefore, by \eqref{EqTheoLangenbruchFz2}, we deduce
\begin{equation}\label{EqTheoIneqLogf} 
\log|G(z)|=\log|f(w)| \geq \widetilde{C_2}\omega\big(\big|\sqrt{z_1^2 + \ldots + z_d^2}\big|\big) = \widetilde{C_2}\omega\big(\sqrt{\big|z_1^2 + \ldots + z_d^2\big|}\big), 
\end{equation} 
if $|\Im(z)| < \varepsilon|\Re(z)|$. Now, from Definition~\ref{DefinitionWeightFunction}($\alpha$) and by the continuity of $G$, it is easy to see that there are constants $C_{2}, C_{3}, C_{4}>0$ such that
\begin{equation}\label{EqTheoGz11} 
\log|G(z)| \geq C_2\omega(z) - C_4,
\end{equation} 
for $|\Im(z)| \leq C_3\big(|\Re(z)| + 1\big).$ 
\end{proof}

\begin{prop}\label{PropEstimatesUltradifferentialOperator} 
Let $G \in \mathcal{H}\big(\mathbb{C}^{d}\big)$ be the function obtained in Theorem~\ref{TheoLangenbruchSeveralVariables}. Then  the function $\displaystyle q(\xi) := \frac1{G(\xi)}$, $\xi \in \mathbb{R}^{d}$, satisfies 
\begin{equation} \label{EqEstimationDerivativeQ} 
\big| D^{\beta}q(\xi) \big| \leq C\beta! R^{-|\beta|} e^{-K\omega(\xi)},
\end{equation} 
for some constants $C,K,R>0$ and every multi-index $\beta \in \mathbb{N}_{0}^{d}$ and every $\xi \in \mathbb{R}^{d}$. 
\end{prop} 
\begin{proof} 
First, we observe that if we take the polyradius $r=(R,\ldots,R) \in \mathbb{R}^{d}_{+}$, with $\displaystyle R \leq \frac1{\sqrt{d}}C_3$ then 
 the polydisc $P(\xi,r)$ satisfies \[ P(\xi,r) \subseteq \widetilde{U}:= \big\{ z \in \mathbb{C}^{d}: |\Im(z)| \leq C_3(|\Re(z)| + 1) \big\}, \] where $\widetilde{U}$ and $C_3>0$ are taken from Theorem~\ref{TheoLangenbruchSeveralVariables} $ii$'). 

Now, we fix a multi-index $\beta \in \mathbb{N}_{0}^{d}$. By taking $\widetilde{C}=\exp\{C_4\}$, where $C_4>0$ comes from Theorem~\ref{TheoLangenbruchSeveralVariables}, and Cauchy's inequalities, we have 
\begin{equation*}\label{EqEstimationQ2} 
|D^{\beta}q(\xi)| \leq \frac{\beta!}{r^{\beta}} \sup_{\zeta \in \partial P(\xi,r)} |q(\zeta)| \leq \widetilde{C}\frac{\beta!}{r^{\beta}} \sup_{\zeta \in \partial P(\xi,r)} e^{-C_2\omega(\zeta)}. 
\end{equation*} 
 Now, since the weight $\omega$ is increasing and satisfies $(\alpha)$,  it is not difficult to see that 
 $$
 -C_{2}\omega(\zeta)\le -K\omega(\xi)+A,
 $$
 where $K,A>0$ only depend on $C_{2}, C_{3}$, the weight $\omega$ and the dimension $d.$ Moreover, $r^{\beta}=R^{|\beta|}$, so we obtain \eqref{EqEstimationDerivativeQ} for $\xi \in \mathbb{R}^d$, which finishes the proof. 
\end{proof} 

In what follows, we will  consider a suitable power of the function of Proposition~\ref{PropEstimatesUltradifferentialOperator}.  The following result can be proved in the same way. 
\begin{cor}\label{CorEstimatesUltradifferentialOperator} 
For $n \in \mathbb{N}$, let $G^{n}$ denote the $n$-th power of the entire function $G$ of Proposition~\ref{PropEstimatesUltradifferentialOperator}. Then $\displaystyle q^{n} = G^{-n}$ satisfies 
\begin{equation}\label{EqEstimationDerivativesUltraOperator} 
|D^{\beta}q^{n}(\xi)| \leq C^{n} \beta! R^{-|\beta|}e^{-nK\omega(\xi)}, 
\end{equation} 
for the same constants $C,K,R>0$ from Proposition~\ref{PropEstimatesUltradifferentialOperator} and for every $\beta \in \mathbb{N}_{0}^{d}$ and  $\xi \in \mathbb{R}^{d}$. 
\end{cor} 

Moreover, we see that there is a constant $C>0$ such that 
\begin{equation}\label{obs1}
 \big| D^{\alpha}G(0) \big| \leq \alpha! e^{C} e^{-C \varphi^{\ast}\big(\frac{|\alpha|}{C}\big)}, 
\end{equation} 
 for all $\alpha \in \mathbb{N}_{0}^{d}$. To prove this, we fix $r=(R,\ldots,R) \in \mathbb{R}^{d}_{+}$ with $R>0$ and $\alpha \in \mathbb{N}_{0}^{d}$. By Cauchy's integral formula we obtain 
\[ D^{\alpha}G(0) = \frac{\alpha!}{(2\pi i)^d} \int_{\partial P(0,r)} \frac{G(\xi)}{\xi^{\alpha+(1,\ldots,1)}} d\xi. \] 
Hence,  
\begin{eqnarray*} 
\big| D^{\alpha}G(0) \big| &\leq& \frac{\alpha!}{(2\pi)^{d}} \frac{(2\pi)^{d} R^{d}}{R^{|\alpha|+d}} \max_{\zeta \in \partial P(0,r)} |G(\zeta)| \leq \frac{\alpha!}{R^{|\alpha|}} e^{\omega(R)+C_1}, 
\end{eqnarray*} where $C_1>0$ comes from $i')$ of Theorem~\ref{TheoLangenbruchSeveralVariables}. Besides, we have 
\begin{eqnarray*} 
\inf_{R>0} \big\{ R^{-|\alpha|} e^{\omega(R)} \big\} &=& \big( \sup_{R>0} \big\{ R^{|\alpha|} e^{-\omega(R)} \big\} \big)^{-1} \le \big( \sup_{s>0} \big\{ e^{s|\alpha|} e^{-\varphi(s)} \big\} \big)^{-1}\\ 
&=& \big( e^{ \sup_{s>0} \{ s|\alpha|- \varphi(s) \} } \big)^{-1}= e^{-\varphi^{\ast}(|\alpha|)}. 
\end{eqnarray*} This implies 
\[ \big| D^{\alpha}G(0) \big| \leq \alpha! e^{C_1-\varphi^{\ast}(|\alpha|)}. \] 
Then, we can take $C:=\max\{C_{1},1\}$ to obtain \eqref{obs1}. 

Since $G$ is entire, we can write $ G(z) = \sum_{\alpha \in \mathbb{N}_{0}^{d}} a_{\alpha} z^{\alpha}, \ z \in \mathbb{C}^{d}$, for some sequence $(a_{\alpha})_{\alpha \in \mathbb{N}_{0}^d} \subseteq \mathbb{C}$. Hence, we also have
\begin{equation}\label{EqEstimationAAlpha} 
|a_{\alpha}| \leq e^{C} e^{-C \varphi^{\ast}\big(\frac{|\alpha|}{C}\big)}, \qquad \alpha \in \mathbb{N}_{0}^{d}. 
\end{equation} 

If $n\in\N$ and we consider the $n$-th power of $G$, $G^{n}$, we also have $G^{n}(z) = \sum_{\alpha \in \mathbb{N}_{0}^{d}} b_{\alpha}z^{\alpha}$, $z \in \mathbb{C}^{d}$, for some sequence $(b_{\alpha})_{\alpha \in \mathbb{N}_{0}^{d}} \subseteq \mathbb{C}$; proceeding as before we can see that 
\begin{equation}\label{EqEstimationBAlpha} 
|b_{\alpha}| \leq e^{nC} e^{-nC \varphi^{\ast}\big(\frac{|\alpha|}{nC}\big)}, \qquad \alpha \in \mathbb{N}_{0}^{d}. 
\end{equation}

\section{Pseudodifferential operators}

\label{section-pseudos} 
Following Prangoski~\cite{Pran2013pseudo} and Shubin~\cite{Shu2001pseudo} we state our definition of global symbol and global amplitude. In what follows, $m\in\R$ and $0<\rho\le 1$.
\begin{defin}\label{symbol} 
A symbol in $\GS_{\rho}^{m,\omega}$ is a function $p(x,\xi) \in C^{\infty}\big(\mathbb{R}^{2d}\big)$ such that for all $n \in \mathbb{N}$, there exists $C_n>0$ with 
\[ \big| D^{\alpha}_x D^{\beta}_{\xi} p(x,\xi) \big| \leq C_n \frac1{\langle(x,\xi)\rangle^{\rho|\alpha+\beta|}} e^{n \rho \varphi^{\ast}\big(\frac{|\alpha+\beta|}{n}\big)} e^{m\omega(x)} e^{m\omega(\xi)}, \] for all $(x,\xi) \in \mathbb{R}^{2d}$ and $(\alpha,\beta) \in \mathbb{N}_0^{2d}$. 
\end{defin} 

\begin{defin}\label{amplitude} 
An amplitude in $\GA_{\rho}^{m,\omega}$ is a function $a(x,y,\xi) \in C^{\infty}\big(\mathbb{R}^{3d}\big)$ such that for all $n\in \mathbb{N}$ there exists $C_n>0$ with 
\[ \big| D^{\alpha}_{x} D^{\beta}_{y} D^{\gamma}_{\xi} a(x,y,\xi) \big| \leq C_n \frac{\langle x-y \rangle^{\rho|\alpha + \beta + \gamma|}}{\langle(x,y,\xi)\rangle^{\rho|\alpha + \beta + \gamma|}} e^{n \rho \varphi^{\ast}\big(\frac{|\alpha + \beta + \gamma|}{n}\big)} e^{m(\omega(x) + \omega(y) + \omega(\xi) )}, \] for all $(x,y,\xi) \in \mathbb{R}^{3d}$ and $(\alpha,\beta,\gamma) \in \mathbb{N}_{0}^{3d}$. 
\end{defin} 

We define the pseudodifferential operators for amplitudes as in Definition~\ref{amplitude} using oscillatory integrals. Let $\chi \in \mathcal{S}_{\omega}\big(\mathbb{R}^{2d}\big)$ be such that $\chi(0,0)=1$. We consider for $f \in \Sch_\omega(\R^{d})$ the double integral 
\begin{equation} 
A_{\delta,\chi}(f)(x) := \int_{\mathbb{R}^d} \int_{\mathbb{R}^d} e^{i(x-y)\xi} a(x,y,\xi) \chi(\delta x, \delta \xi) f(y) dy d\xi.\label{oscilatoria} 
\end{equation} 
We will see that $A_{\delta,\chi}(f)$  converges for every $f\in\Sch_{\omega}(\R^{d})$ when $\delta \to 0$, defining a linear and continuous operator $A:\Sch_{\omega}(\R^{d})\to \Sch_{\omega}(\R^{d})$ given by the iterated integral 
\[ A(f) = \int_{\mathbb{R}^d} \left(\int_{\mathbb{R}^d} e^{i(x-y)\xi} a(x,y,\xi) f(y) dy\right) d\xi, \qquad f \in \Sch_\omega(\R^{d}).\] 

\begin{prop}\label{LemmaIteratedIntegralCauchy} 
Let $\chi\in \mathcal{S}_{\omega}\big(\mathbb{R}^{2d}\big)$. Then, for any function $f\in\Sch_{\omega}(\R^{d})$, the sequence $\big( A_{\frac1{n},\chi}(f)\big)_{n\in \N}$ as in \eqref{oscilatoria} is a Cauchy sequence in $\Sch_\omega(\R^{d})$. 
\end{prop} 
\begin{proof} 
We consider  the family of seminorms of Remark~\ref{Notaflambda}. We show that, for any $f\in\Sch_{\omega}(\R^{d})$ and $\lambda>0$, $$|(A_{\frac1{k},\chi} - A_{\frac1{l},\chi})(f)|_{\lambda}$$ goes to zero when $l,k$ tend to infinity. 

To this aim, we fix $\beta \in \mathbb{N}_0^d$ and $x \in \mathbb{R}^d$, and calculate 
\begin{eqnarray} \label{integrand}
\lefteqn{ D^{\beta}_x \int_{\mathbb{R}^d} \int_{\mathbb{R}^d} e^{i(x-y)\xi} a(x,y,\xi) \Big( \chi\big(\frac1{k} x, \frac1{k} \xi\big) - \chi\big(\frac1{l}x, \frac1{l}\xi \big) \Big) f(y) dy d\xi  }\nonumber\\ 
&&= \sum_{\beta_1+\beta_2+\beta_3=\beta} \frac{\beta!}{\beta_1! \beta_2! \beta_3!} \iint_{\mathbb{R}^{2d}} e^{i(x-y)\xi} \xi^{\beta_1} D^{\beta_2}_x a(x,y,\xi) D^{\beta_3}_x \Big( \chi\big(\frac1{k} x, \frac1{k} \xi\big) - \chi\big(\frac1{l}x, \frac1{l}\xi \big) \Big) f(y) dy d\xi.\ \ \ \ \  
\end{eqnarray} 
For the ultradifferential operator $G(D)$ and its corresponding symbol $G(\xi)$ given in Theorem~\ref{TheoLangenbruchSeveralVariables}, the following formula holds for each $n\in\N$: 
\begin{equation}\label{ultrafor} 
e^{i(x-y)\xi} = \frac1{G^{n}(y-x)} G^{n}(-D_{\xi}) \Big( \frac1{G^{n}(\xi)} G^{n}(-D_y) e^{i(x-y)\xi} \Big).
\end{equation} 
Now, we use the notation of \eqref{EqEstimationBAlpha} and formula \eqref{ultrafor}, and integrate by parts to obtain the following expression for the integrand of~\eqref{integrand}:
\begin{eqnarray*} 
&& e^{i(x-y)\xi} \frac1{G^n(\xi)} G^n(D_y) \Big( \frac1{G^n(y-x)} G^n(D_{\xi}) \Big( \xi^{\beta_1} D^{\beta_2}_x a(x,y,\xi) \times \\ 
&&\ \ \ \ \times D^{\beta_3}_x \Big( \chi\big(\frac1{k} x, \frac1{k} \xi\big) - \chi\big(\frac1{l}x, \frac1{l}\xi \big) \Big) f(y) \Big) \Big)  \\ 
&& = e^{i(x-y)\xi} \frac1{G^n(\xi)} G^n(D_y) \Big( \frac1{G^n(y-x)} \sum_{\tau \in \N_0^{d}} b_{\tau} \sum_{\tau_1+\tau_2+\tau_3=\tau,\,\tau_{1}\le \beta_{1}} \frac{\tau!}{\tau_1! \tau_2! \tau_3!} \frac{\beta_1!}{(\beta_1-\tau_1)!} (-i)^{|\tau_{1}|} \xi^{\beta_1-\tau_1} \times \\ 
&&\ \ \ \  \times D^{\beta_2}_x D^{\tau_2}_{\xi} a(x,y,\xi) D^{\beta_3}_x D^{\tau_3}_{\xi} \Big( \chi\big(\frac1{k} x, \frac1{k} \xi\big) - \chi\big(\frac1{l}x, \frac1{l}\xi \big) \Big) f(y) \Big) \\ 
&& = e^{i(x-y)\xi} \frac1{G^n(\xi)} \sum_{\epsilon,\tau \in \N_0^{d}} b_{\epsilon} b_{\tau} \sum_{\substack{\epsilon_1+\epsilon_2+\epsilon_3=\epsilon \\ \tau_1+\tau_2+\tau_3=\tau,\,\tau_{1}\le \beta_{1}}} \frac{\epsilon!}{\epsilon_1! \epsilon_2! \epsilon_3!} \frac{\tau!}{\tau_1! \tau_2! \tau_3!} \frac{\beta_1!}{(\beta_1-\tau_1)!} (-i)^{|\tau_{1}|}\xi^{\beta_1-\tau_1} \times \\ 
&&\ \ \ \  \times D^{\epsilon_1}_y \frac1{G^n(y-x)} D^{\beta_2}_x D^{\epsilon_2}_y D^{\tau_2}_{\xi} a(x,y,\xi) D^{\beta_3}_x D^{\tau_3}_{\xi} \Big( \chi\big(\frac1{k} x, \frac1{k} \xi\big) - \chi\big(\frac1{l}x, \frac1{l}\xi \big) \Big) D^{\epsilon_3}_y f(y).
\end{eqnarray*} 
Hence,  $D^{\beta}_x\big(A_{\frac1{k},\chi} - A_{\frac1{l},\chi}\big)(f)$ is equal to 
\begin{eqnarray*} 
&& \sum_{\epsilon,\tau \in \N_0^{d}} b_{\epsilon} b_{\tau} \sum_{\substack{\epsilon_1+\epsilon_2+\epsilon_3=\epsilon \\ \tau_1+\tau_2+\tau_3=\tau, \,\tau_{1}\le \beta_{1} \\ \beta_1+\beta_2+\beta_3=\beta}} \frac{\epsilon!}{\epsilon_1! \epsilon_2! \epsilon_3!} \frac{\tau!}{\tau_1! \tau_2! \tau_3!} \frac{\beta!}{\beta_1! \beta_2! \beta_3!} \frac{\beta_1!}{(\beta_1-\tau_1)!} \int_{\mathbb{R}^d} \int_{\mathbb{R}^d} e^{i(x-y)\xi} \frac{(-i)^{|\tau_{1}|}}{G^n(\xi)} \xi^{\beta_1-\tau_1} \times \\ 
&& \ \ \ \ \times D^{\epsilon_1}_y \frac1{G^n(y-x)} D^{\beta_2}_x D^{\epsilon_2}_y D^{\tau_2}_{\xi} a(x,y,\xi) D^{\beta_3}_x D^{\tau_3}_{\xi} \Big( \chi\big(\frac1{k} x, \frac1{k} \xi\big) - \chi\big(\frac1{l}x, \frac1{l}\xi \big) \Big) D^{\epsilon_3}_y f(y) dy d\xi. 
\end{eqnarray*} 

Now, we fix $\lambda>0$ and take $s \geq \lambda$ and $n\in\N$ to be determined. Since $f \in \Sch_\omega(\R^{d})$, for the constant $L>0$ of Lemma~\ref{lemmafistrella}\,(1) we have 
\[ |D^{\epsilon_3}_y f(y)| \leq E'_s e^{sL^3 \varphi^{\ast}\big(\frac{|\epsilon_3|}{sL^3}\big)} e^{-sL^3\omega(y)}. \] Moreover, by the definition of amplitude and according to Lemma~\ref{LemmaCorcheteIneq} and formula \eqref{EqLLambdaVarphi}, we have that there is a constant $E_{s}>0$ depending on $s$ such that 
\begin{eqnarray*} 
|D^{\beta_2}_x D^{\epsilon_2}_y D^{\tau_2}_{\xi} a(x,y,\xi)| &\leq& E_{s} \Big(\frac{\langle x-y \rangle}{\langle(x,y,\xi)\rangle}\Big)^{\rho|\beta_2+\epsilon_2+\tau_2|} e^{4L^4s\rho \varphi^{\ast}\big(\frac{|\beta_2+\epsilon_2+\tau_2|}{4L^4s}\big)} e^{m\omega(x)} e^{m\omega(y)} e^{m\omega(\xi)} \\ 
&\leq& E_{s} \sqrt{2}^{|\beta_2+\epsilon_2+\tau_2|} e^{4L^4s \varphi^{\ast}\big(\frac{|\beta_2+\epsilon_2+\tau_2|}{4L^4s}\big)} e^{m\omega(x)} e^{m\omega(y)} e^{m\omega(\xi)}\\ 
&\leq& E_{s} e^{4L^4s} e^{4L^3s \varphi^{\ast}\big(\frac{|\beta_2+\epsilon_2+\tau_2|}{4L^3s}\big)} e^{m\omega(x)} e^{m\omega(y)} e^{m\omega(\xi)}. 
\end{eqnarray*} 
By \eqref{EqLemmaTechnical1} and \eqref{EqOmegaCorchetex}, we also obtain
\[ |\xi|^{|\beta_1-\tau_1|} \leq \langle \xi \rangle^{|\beta_1-\tau_1|} \leq e^{\lambda L^3 \varphi^{\ast}\big(\frac{|\beta_1-\tau_1|}{\lambda L^3}\big)} e^{\lambda L^3 \omega(\langle \xi \rangle)} \leq e^{\lambda L^3 \varphi^{\ast}\big(\frac{|\beta_1|}{\lambda L^3}\big)} e^{\lambda L^4} e^{\lambda L^4\omega(\xi)}. \] By \eqref{EqEstimationBAlpha},  there is $C_1>0$ that depends only on $G$ such that 
\[ |b_{\epsilon}| \leq e^{nC_1} e^{-nC_1 \varphi^{\ast}\big(\frac{|\epsilon|}{nC_1}\big)}, \qquad |b_{\tau}| \leq e^{nC_1} e^{-nC_1 \varphi^{\ast}\big(\frac{|\tau|}{nC_1}\big)}, \] 
and, by Corollary~\ref{CorEstimatesUltradifferentialOperator} and Proposition~\ref{PropTechnical}, there are constants $C_2,C_3,C_4>0$ which depend only on $G$ such that 
\begin{eqnarray*} 
\Big| \frac1{G^n(\xi)} \Big| &\leq& C_4^n e^{-nC_2 \omega(\xi)},\ \ \ \ \mbox{and} \\ 
\Big| D^{\epsilon_1}_y \frac1{G^n(y-x)} \Big| &\leq& C_4^n \epsilon_1! C_3^{-|\epsilon_1|} e^{-nC_2 \omega(y-x)} \leq C_4^n C_s e^{sL^3 \varphi^{\ast}\big(\frac{|\epsilon_1|}{sL^3}\big)} e^{-nC_2 \omega(y-x)}, 
\end{eqnarray*} where $C_{s}>0$ depends on $C_3$ and $s$. Finally, since $\omega(x) \leq L\omega(y-x) + L\omega(y) + L$, and (again by Proposition~\ref{PropTechnical}),
\[ \frac{\beta_1!}{(\beta_1-\tau_1)!} \leq 2^{|\beta_1|} \tau_1! \leq 2^{|\beta_1|} C'_s e^{sL^3 \varphi^{\ast}\big(\frac{|\tau_1|}{sL^3}\big)}, \] for some constant $C'_s>0$ depending on $s$, we get 
\begin{eqnarray*} 
\lefteqn{\big|D^{\beta}_{x} \big((A_{\frac1{k},\chi} - A_{\frac1{l},\chi} )(f)\big)(x)\big|}\label{series}\\ 
&&\nonumber\le \sum_{\epsilon,\tau \in \N_0^{d}} e^{2nC_1} e^{-nC_1 \varphi^{\ast}\big(\frac{|\epsilon|}{nC_1}\big)} e^{-nC_1 \varphi^{\ast}\big(\frac{|\tau|}{nC_1}\big)} \sum_{\substack{\epsilon_1+\epsilon_2+\epsilon_3=\epsilon \\ \tau_1+\tau_2+\tau_3=\tau, \,\tau_{1}\le \beta_{1} \\ \beta_1+\beta_2+\beta_3=\beta}} \frac{\epsilon!}{\epsilon_1! \epsilon_2! \epsilon_3!} \frac{\tau!}{\tau_1! \tau_2! \tau_3!} \frac{\beta!}{\beta_1! \beta_2! \beta_3!} 2^{|\beta_1|} C'_{s} \times \\ 
&& \nonumber\ \ \ \ \times e^{sL^3 \varphi^{\ast}\big(\frac{|\tau_1|}{sL^3}\big)} \int \Big| D^{\beta_3}_x D^{\tau_3}_{\xi} \Big( \chi\big(\frac1{k}x, \frac1{k}\xi \big) - \chi\big(\frac1{l}x, \frac1{l}\xi\big) \Big) \Big| \left[\int C^{2n}_4 e^{-nC_2 \omega(\xi)} e^{\lambda L^3 \varphi^{\ast}\big(\frac{|\beta_1|}{\lambda L^3}\big)} e^{\lambda L^4} \right. \times \\ 
&& \nonumber\ \ \ \ \times e^{\lambda L^4 \omega(\xi)} C_{s} e^{sL^3 \varphi^{\ast}\big(\frac{|\epsilon_1|}{sL^3}\big)} e^{-nC_2 \omega(y-x)} E_s e^{4L^4s} e^{4L^3s \varphi^{\ast}\big(\frac{|\beta_2+\epsilon_2+\tau_2|}{4L^3s}\big)} \times \\ 
&& \nonumber \ \ \ \ \left.\times e^{m\omega(y)} e^{mL \omega(y-x)} e^{mL \omega(y)} e^{mL} e^{m\omega(\xi)} E'_s e^{sL^3 \varphi^{\ast}\big(\frac{|\epsilon_3|}{sL^3}\big)} e^{-sL^3 \omega(y)} dy \right]d\xi. 
\end{eqnarray*} 
We set 
\[ n \geq \max\big\{ \frac1{C_2}(1+\lambda L^4+m), \frac1{C_2}(m+\lambda)L \big\}. \] The first one is stated in order to get \[ e^{(m+\lambda L^4-nC_2)\omega(\xi)} \leq e^{-\omega(\xi)}, \] and the other one to obtain \[ e^{(mL-nC_2)\omega(y-x)} \leq e^{-\lambda L \omega(y-x)}. \] 
Moreover, we put 
\[ s \geq \max\big\{ \frac1{L^3}(1+\lambda L+m+mL), nC_1\}. \] In this case, by the first inequality we obtain 
\[ 
e^{(m+mL-sL^3)\omega(y)} \leq e^{-\omega(y)} e^{-\lambda L \omega(y)}. 
\] 
According to $\omega(x) \leq L\omega(y-x) + L\omega(y) + L$, we get
\[ e^{-\lambda L \omega(y-x)} e^{-\lambda L \omega(y)} \leq e^{\lambda L} e^{-\lambda \omega(x)}. \] 
By the mean value theorem, there exists $c$ in the line segment between $(\frac1{l}x, \frac1{l}\xi)$ and $(\frac1{k}x, \frac1{k}\xi)$ such that 
\begin{eqnarray*} 
\big| D^{\beta_3}_x D^{\tau_3}_{\xi}\Big(\chi\big(\frac1{k}x, \frac1{k}\xi \big) - \chi\big(\frac1{l}x, \frac1{l}\xi \big)\Big)\big|& = &|\nabla D^{\beta_3}_x D^{\tau_3}_{\xi}\chi(c)| \big| \frac1{k} - \frac1{l}\big||(x,\xi)|\\ &\leq& D'_{\lambda,s} e^{\lambda \varphi^{\ast}\big(\frac1{\lambda}\big)+2sL^{3}\lambda \varphi^{\ast}\big(\frac{|\beta_{3}+\tau_{3}|}{2sL^{3}\lambda}\big)} \big| \frac1{k} - \frac1{l} \big||x||\xi|, 
\end{eqnarray*} for some constant $D'_{\lambda,s}>0$. 
Now, by Lemma~\ref{lemmafistrella}, since $s\ge \lambda,$ we have 
\begin{eqnarray*} 
&& 2^{|\beta_1|} e^{sL^3 \varphi^{\ast}\big(\frac{|\tau_1|}{sL^3}\big)} e^{\lambda L^3 \varphi^{\ast}\big(\frac{|\beta_1|}{\lambda L^3}\big)} e^{sL^3 \varphi^{\ast}\big(\frac{|\epsilon_1|}{sL^3}\big)} e^{4sL^3 \varphi^{\ast}\big(\frac{|\beta_2+\epsilon_2+\tau_2|}{4sL^3}\big)} e^{sL^3 \varphi^{\ast}\big(\frac{|\epsilon_3|}{sL^3}\big)} e^{2sL^3 \varphi^{\ast}\big(\frac{|\beta_3+\tau_3|}{2sL^3}\big)}  \\ 
&& \quad\leq e^{\lambda L^3} e^{\lambda L^2 \varphi^{\ast}\big(\frac{|\beta|}{\lambda L^2}\big)} e^{sL^3\varphi^{\ast}\big(\frac{|\tau|}{sL^3}\big)} e^{sL^3\varphi^{\ast}\big(\frac{|\epsilon|}{sL^3}\big)}. 
\end{eqnarray*} 
Since the selection of $n$ and $s$ depends on $\lambda$, we get this new estimate, for a constant $C'_{\lambda}>0$: 
\begin{eqnarray} 
\lefteqn{\big|D^{\beta}_{x} \big((A_{\frac1{k},\chi} - A_{\frac1{l},\chi} )(f)\big)(x)\big|}\label{seriess}\\ 
&& \nonumber\le C'_{\lambda} \Big|\frac1{k}-\frac1{l}\Big| \sum_{\epsilon,\tau \in \N_0^{d}} e^{-nC_1 \varphi^{\ast}\big(\frac{|\epsilon|}{nC_1}\big)} e^{-nC_1 \varphi^{\ast}\big(\frac{|\tau|}{nC_1}\big)} \sum_{\substack{\epsilon_1+\epsilon_2+\epsilon_3=\epsilon \\ \tau_1+\tau_2+\tau_3=\tau, \,\tau_{1}\le \beta_{1} \\ \beta_1+\beta_2+\beta_3=\beta}} \frac{\epsilon!}{\epsilon_1! \epsilon_2! \epsilon_3!} \frac{\tau!}{\tau_1! \tau_2! \tau_3!} \frac{\beta!}{\beta_1! \beta_2! \beta_3!} \times \\ 
&& \nonumber\ \ \ \times e^{\lambda L^2 \varphi^{\ast}\big(\frac{|\beta|}{\lambda L^2}\big)} e^{sL^3\varphi^{\ast}\big(\frac{|\tau|}{sL^3}\big)} e^{sL^3\varphi^{\ast}\big(\frac{|\epsilon|}{sL^3}\big)}|x| e^{-\lambda \omega(x)} \Big( \int e^{-\omega(y)} dy \Big) \Big( \int |\xi| e^{-\omega(\xi)} d\xi \Big). 
\end{eqnarray} 
Again by Lemma~\ref{lemmafistrella}, using multinomial coefficients, we obtain 
\begin{eqnarray*} 
&&\sum_{\substack{\epsilon_1+\epsilon_2+\epsilon_3=\epsilon \\ \tau_1+\tau_2+\tau_3=\tau,\,\tau_{1}\le \beta_{1} \\ \beta_1+\beta_2+\beta_3=\beta}} \frac{\epsilon!}{\epsilon_1! \epsilon_2! \epsilon_3!} \frac{\tau!}{\tau_1! \tau_2! \tau_3!} \frac{\beta!}{\beta_1! \beta_2! \beta_3!} e^{\lambda L^2 \varphi^{\ast}\big(\frac{|\beta|}{\lambda L^2}\big)} e^{sL^3\varphi^{\ast}\big(\frac{|\tau|}{sL^3}\big)} e^{sL^3\varphi^{\ast}\big(\frac{|\epsilon|}{sL^3}\big)} \\ 
&&\ \ \ \ \ \ \ \ \quad =3^{|\epsilon+\tau+\beta|}e^{\lambda L^2 \varphi^{\ast}\big(\frac{|\beta|}{\lambda L^2}\big)} e^{sL^3\varphi^{\ast}\big(\frac{|\tau|}{sL^3}\big)} e^{sL^3\varphi^{\ast}\big(\frac{|\epsilon|}{sL^3}\big)}\\ 
&&\ \ \ \ \ \ \ \ \quad \leq e^{(\lambda +2sL)(L + L^2)} e^{\lambda \varphi^{\ast}\big(\frac{|\beta|}{\lambda}\big)} e^{sL \varphi^{\ast}\big(\frac{|\tau|}{sL}\big)} e^{sL \varphi^{\ast}\big(\frac{|\epsilon|}{sL}\big)}. 
\end{eqnarray*} 

Now, we see that the series in \eqref{series} converge. We treat the sum in $\epsilon$. Since $s \geq nC_1$, we have, for each $\epsilon\in\N_{0}^{d},$ by \eqref{EqLLambdaVarphi}, 
\begin{eqnarray} \label{series1}
e^{-nC_1 \varphi^{\ast}\big(\frac{|\epsilon|}{nC_1}\big)} e^{sL \varphi^{\ast}\big(\frac{|\epsilon|}{sL}\big)} 
&=& \Big(\frac1{e}\Big)^{|\epsilon|} e^{-nC_1 \varphi^{\ast}\big(\frac{|\epsilon|}{nC_1}\big)} e^{|\epsilon| + sL \varphi^{\ast}\big(\frac{|\epsilon|}{sL}\big)} \\ 
&\leq& e^{sL} \Big(\frac1{e}\Big)^{|\epsilon|} e^{-nC_1 \varphi^{\ast}\big(\frac{|\epsilon|}{nC_1}\big)} e^{s \varphi^{\ast}\big(\frac{|\epsilon|}{s}\big)} \leq e^{sL} \Big(\frac1{e}\Big)^{|\epsilon|}. \nonumber
\end{eqnarray} 
By formula \cite[(0.3.16)]{Nic_Rod2010global}, we have 
\[ \# \{ \epsilon \in \mathbb{N}_0^d: |\epsilon| = j\} = \binom{j+d-1}{d-1}. \] 
Then, we deduce 
\begin{eqnarray} \label{series2}
\sum_{\epsilon \in\N_{0}^{d}} \Big(\frac1{e}\Big)^{|\epsilon|} &=& \sum_{j=0}^{\infty} \sum_{|\epsilon|=j} \Big(\frac1{e}\Big)^j = \sum_{j=0}^{\infty} \Big(\frac1{e}\Big)^j \binom{j+d-1}{d-1} \\ 
&\leq&  2^{d-1} \sum_{j=0}^{\infty} \Big(\frac{2}{e}\Big)^j < +\infty.\nonumber 
\end{eqnarray} 
The convergence of the series in $\tau$ follows in the same way. 

Finally, we get 
\begin{eqnarray}\label{case2} 
\big|D^{\beta}_{x} \big((A_{\frac1{k},\chi} - A_{\frac1{l},\chi})(f)\big)(x)\big| &\leq& C_{\lambda} \big| \frac1{k} - \frac1{l}\big| \Big( \sum_{j=0}^{\infty} \big(\frac{2}{e}\big)^j \Big)^2 e^{\lambda \varphi^{\ast}\big( \frac{|\beta|}{\lambda}\big)} \times \\ 
&&\ \ \ \ \times |x| e^{-\lambda \omega(x)} \Big( \int e^{-\omega(y)} dy \Big) \Big( \int |\xi| e^{-\omega(\xi)} d\xi \Big),\nonumber 
\end{eqnarray} for some constant $C_{\lambda}>0$ depending on $\lambda$.

From \eqref{case2} we conclude that 
$$ 
\Big|\big(A_{\frac1{k},\chi} - A_{\frac1{l},\chi} \big)(f)\Big|_{\lambda}\longrightarrow 0, \mbox{ as }k,l\to +\infty, 
$$ 
for each $\lambda>0$ and, hence, $\{A_{\frac1{n},\chi}(f)\}_{n\in\N}$ is a Cauchy sequence in $\Sch_{\omega}(\R^{d}).$ 
\end{proof} 

\begin{lema}\label{previopseudo} 
Given an amplitude $a(x,y,\xi)\in \GA_{\rho}^{m,\omega}$ and $f\in\Sch_{\omega}(\R^{d})$, for each $\lambda>0$ there is $C_{\lambda}>0$ such that for all $x,\xi\in\R^{d}$, we have 
\[ \big| \int_{\mathbb{R}^d} e^{i(x-y)\xi} a(x,y,\xi) f(y) dy \big| \leq C_{\lambda} e^{-\lambda \omega(\xi)} e^{m\omega(x)}. \] 
\end{lema} 

\begin{proof} 
We follow the ideas of the proof of Proposition~\ref{LemmaIteratedIntegralCauchy} and use a suitable integration by parts in the integral \begin{equation} 
\int e^{i(x-y)\xi} a(x,y,\xi) f(y) dy.\label{integral-en-y} 
\end{equation} 
Here, we consider the formula \[ e^{i(x-y)\xi} = \frac1{G(\xi)} G(-D_y) e^{i(x-y)\xi},\] 
which is also true for a suitable power of $G(D)$, say $G^n(D)$, with $n \in \mathbb{N}$ to be determined. Integration by parts yields that the integrand in \eqref{integral-en-y} is equal to 
\begin{eqnarray*} 
\lefteqn{e^{i(x-y)\xi} \frac1{G^n(\xi)} G^n(D_y) \Big( a(x,y,\xi) f(y) \Big) }\\ 
&& = e^{i(x-y)\xi} \frac1{G^n(\xi)} \sum_{\tau \in\N_{0}^{d}} b_{\tau} \sum_{\tau_1+\tau_2=\tau} \frac{\tau!}{\tau_1! \tau_2!} D^{\tau_1}_y a(x,y,\xi) D^{\tau_2}_y f(y). 
\end{eqnarray*} 
Now, proceeding in a similar way to that of Proposition~\ref{LemmaIteratedIntegralCauchy} we get the conclusion. 
\end{proof} 

Applying the definition of amplitude we  show the following

\begin{lema}\label{lema-BS}
Given an amplitude $a(x,y,\xi) \in \GA^{m,\omega}_{\rho}$ and $\chi \in \mathcal{S}_{\omega}\big(\mathbb{R}^{2d}\big)$, we denote
\[ K(x,y) := \int_{\mathbb{R}^d} e^{i(x-y)\xi} a(x,y,\xi) \chi(x,\xi) d\xi. \] We have
\begin{enumerate}
\item[a)] $K(x,y) \in \mathcal{S}_{\omega}\big(\mathbb{R}^{2d}\big)$.
\item[b)] The linear operator $T: \mathcal{S}_{\omega}(\mathbb{R}^d) \to \mathcal{S}_{\omega}(\mathbb{R}^d)$ given by $T(f)(x) = \int K(x,y) f(y) dy$ is continuous.
\end{enumerate}
\end{lema}
\begin{nota}{\rm 
If the function $\chi \in \mathcal{S}_{\omega}(\mathbb{R}^d)$ only depends on $\xi$, we do not obtain a) $K \in \mathcal{S}_{\omega}\big(\mathbb{R}^{2d}\big)$ in the lemma above, but this weaker condition: For every $\lambda>0$ there is $C_{\lambda}>0$ such that for every $\alpha,\beta \in \mathbb{N}_0^d$ and every $x,y \in \mathbb{R}^d$, the function $K \in C^{\infty}\big(\mathbb{R}^{2d}\big)$ and satisfies
\[ |D^{\alpha}_x D^{\beta}_y K(x,y)| \leq C_{\lambda} e^{\lambda \varphi^{\ast}\big(\frac{|\alpha+\beta|}{\lambda}\big)} e^{m\omega(y)}. \] However, this is also sufficient to have that the integral operator $T(f)(x) = \int K(x,y) f(y) dy$ is continuous.}
\end{nota}
\begin{proof}[Proof of  Lemma~\ref{lema-BS}]
a) We fix $\alpha,\beta \in \mathbb{N}_0^d$ and calculate
\[ D^{\alpha}_x D^{\beta}_y K(x,y) = \sum_{\substack{\alpha_1 + \alpha_2 + \alpha_3 = \alpha \\ \beta_1 + \beta_2 = \beta}} \frac{\alpha!}{\alpha_1! \alpha_2! \alpha_3!} \frac{\beta!}{\beta_1! \beta_2!} (-1)^{\beta_1} \int e^{i(x-y)\xi} \xi^{\alpha_1+\beta_1} D^{\alpha_2}_x D^{\beta_2}_{y} a(x,y,\xi) D^{\alpha_3}_x \chi(x,\xi) d\xi. \] As in Proposition~\ref{LemmaIteratedIntegralCauchy}, we perform a suitable integration by parts with the formula
\[ e^{i(x-y)\xi} = \frac1{G^n(y-x)} G^n(-D_{\xi}) e^{i(x-y)\xi}, \] for some power $n \in \mathbb{N}$, to be determined, of the ultradifferential operator given in Theorem~\ref{TheoLangenbruchSeveralVariables}. From now on, the proof follows the lines of that of Proposition~\ref{LemmaIteratedIntegralCauchy}.

b) First, we observe that for $f\in\Sch_{\omega}(\R^{d})$, since $\varphi^{*}(0)=0$, we have, for any $\mu>0$,
$$
\sup_{y\in\R^{d}}|f(y)|\le \sup_{y\in\R^{d}}|f(y)| e^{\mu\omega(y)}\le \sup_{\beta\in\N_{0}^{d}}\sup_{y\in\R^{d}}|f^{(\beta)}(y)|e^{-\mu \varphi^{\ast}\big(\frac{|\beta|}{\mu}\big)}e^{\mu\omega(y)}=|f|_{\mu},
$$
being $|\cdot|_{\mu}$ the seminorm defined in Remark~\ref{Notaflambda}. Now, to prove that the operator $T$ is continuous, we  differentiate under the integral sign the function $T(f)(x)$ to obtain that for all $\lambda>0$, there exists $C_{\lambda}>0$ such that
\begin{eqnarray*} 
\lefteqn{|D^{\alpha}_x T(f)(x)| \leq \int |D^{\alpha}_x K(x,y)| |f(y)| dy }\\&& \le C_{\lambda} e^{\lambda \varphi^{\ast}\big(\frac{|\alpha|}{\lambda}\big)} e^{-\lambda\omega(x)} \int e^{-\lambda\omega(y)} |f(y)| dy\le  C_{\lambda} e^{\lambda \varphi^{\ast}\big(\frac{|\alpha|}{\lambda}\big)} e^{-\lambda\omega(x)} |f|_{\mu} \int e^{-\lambda\omega(y)}  dy,
\end{eqnarray*}
for any $\mu>0$, which gives the conclusion.
\end{proof}


\begin{theo}\label{pseudotheo} 
The operator $A:\Sch_\omega(\R^{d}) \to \Sch_\omega(\R^{d})$ given by the iterated integral
\begin{equation} 
A(f)(x) := \int_{\mathbb{R}^d} \Big(\int_{\mathbb{R}^d} e^{i(x-y)\xi} a(x,y,\xi) f(y) dy \Big) d\xi\label{iterated} 
\end{equation} 
is well defined, linear and continuous. 
\end{theo} 
\begin{proof} 
As in \eqref{oscilatoria}, we fix $\chi\in\Sch_{\omega}(\R^{2d})$ such that $\chi(0,0)=1$. Since $\Sch_{\omega}(\R^{d})$ is a Fr\'echet space, for every $f\in\Sch_{\omega}(\R^{d})$ the sequence $\{ A_{\frac1{n},\chi}(f) \}_{n \in \mathbb{N}}$ converges in $\Sch_{\omega}(\R^{d})$ by Proposition~\ref{LemmaIteratedIntegralCauchy}. Moreover,  the operator $A_{\frac1{n},\chi}:\Sch_{\omega}(\R^{d})\to \Sch_{\omega}(\R^{d})$ is linear and, by Lemma~\ref{lema-BS},   well defined and continuous for every $n\in\N$.  We denote by $A_{\chi}$ the operator given by the limit: 
\[ A_{\chi}(f) := \lim_{n \to \infty} \int_{\mathbb{R}^d} \int_{\mathbb{R}^d} e^{i(x-y)\xi} a(x,y,\xi) f(y) \chi\big(\frac1{n}(x,\xi)\big) dy d\xi, \ \ f\in\Sch_{\omega}(\R^{d}). \] This operator is well defined and linear from $\Sch_{\omega}(\R^{d})$ to $\Sch_{\omega}(\R^{d})$ by Proposition~\ref{LemmaIteratedIntegralCauchy}. Moreover, it is continuous by Banach-Steinhaus theorem. 

Now, we prove formula \eqref{iterated} and, hence, that $I_{\chi}$ does not depend on the selection of $\chi\in \Sch_{\omega}(\R^{2d})$ with $\chi(0,0)=1$. By Lemma~\ref{previopseudo} we have, for all $n\in\N$, 
\begin{eqnarray*} 
\Big| \int e^{i(x-y)\xi} a(x,y,\xi) f(y) \chi\big(\frac1{n}x, \frac1{n}\xi\big) dy \Big| &=& \big|\chi\big(\frac1{n}x, \frac1{n}\xi\big)\big| \Big|\int e^{i(x-y)\xi} a(x,y,\xi) f(y) dy\Big| \\ 
&\leq& C_{\lambda} e^{-\lambda \omega(\xi)} e^{m\omega(x)} \Big(\sup_{\eta \in \mathbb{R}^{2d}} |\chi(\eta)|\Big), 
\end{eqnarray*} which is integrable in $\xi$. Moreover, 
\[ \int e^{i(x-y)\xi} a(x,y,\xi) f(y) \chi\big(\frac1{n}x,\frac1{n}\xi\big) dy \longrightarrow \int e^{i(x-y)\xi} a(x,y,\xi) f(y) dy \] pointwise on $x,\xi \in \mathbb{R}^d$ when $n$ goes to infinity. An application of Lebesgue theorem gives the conclusion. 
\end{proof} 

\begin{defin} 
\label{pseudodef} 
The operator $A$ given in Theorem~\ref{pseudotheo} is called {\em global $\omega$-pseudodifferential operator} associated to the amplitude $a(x,y,\xi)$. 
\end{defin} 

\begin{nota} 
In the hypothesis of Proposition~\ref{LemmaIteratedIntegralCauchy} we can also use a function $\chi \in \mathcal{S}_{\omega}(\mathbb{R}^d)$ which only depends on the variable $\xi$. In the same manner, Theorem~\ref{pseudotheo} is also true if we consider $\chi \in \mathcal{S}_{\omega}(\mathbb{R}^d)$ depending only on $\xi$ that satisfies $\chi(0)=1$. The proofs of both results follow in the same way. 
\end{nota} 

The use of amplitudes permits to extend the operator to the space of ultradistributions in an easy way by duality, as we can see in the next result. We omit its proof since is similar to the one of \cite[Theorem 2.5]{FGJ2005pseudo}.

\begin{prop}\label{LemmaExtensionPseudoDifferentialOperator}
The pseudodifferential operator $A:\mathcal{S}_{\omega}(\R^{d}) \to \mathcal{S}_{\omega}(\R^{d})$ extends to a linear and continuous operator  $\widetilde{A}:\mathcal{S}'_{\omega}(\R^{d})\to \Sch_{\omega}'(\R^{d})$.
\end{prop}

As in \cite[Theorem 2.5]{FGJ2005pseudo}, we observe that for any pseudodifferential operator $A:\mathcal{S}_{\omega}(\R^{d}) \to \mathcal{S}_{\omega}(\R^{d})$ with amplitude $a(x,y,\xi)$, we have that the transpose operator when restricted to $\mathcal{S}_{\omega}(\R^{d}),$ $A^{t}|_{\Sch_\omega(\R^{d})}:\mathcal{S}_{\omega}(\R^{d}) \to \mathcal{S}_{\omega}(\R^{d})$, is a pseudodifferential operator with amplitude $a(y,x,-\xi)$.


The proof of the next result is standard.
\begin{prop}\label{PropExampleExtensionPseudodifferentialOperator}
Let $T:\mathcal{S}_{\omega}(\R^{d}) \to \mathcal{S}_{\omega}(\R^{d})$ be a pseudodifferential operator. The following assertions are equivalent:
\begin{enumerate}
\item[$(1)$] $T$ has a linear and continuous extension $\widetilde{T}:\mathcal{S}'_{\omega}(\R^{d}) \to \mathcal{S}_{\omega}(\R^{d})$;
\item[$(2)$] There exists $K(x,y) \in \mathcal{S}_{\omega}\big(\mathbb{R}^{2d}\big)$ such that
\[ (T\varphi)(x) = \int K(x,y) \varphi(y) dy, \qquad \varphi \in \mathcal{S}_{\omega}(\R^{d}). \]
\end{enumerate}
\end{prop}
\begin{defin}\label{regu}
A pseudodifferential operator $T:\mathcal{S}_{\omega}(\R^{d}) \to \mathcal{S}_{\omega}(\R^{d})$ that satisfies $(1)$ or $(2)$ of Proposition~\ref{PropExampleExtensionPseudodifferentialOperator} is called \emph{$\omega$-regularizing.}
\end{defin}

\begin{exam}\label{examples}{\rm 
(a) As in \cite[Example 2.11]{FGJ2005pseudo}, particular cases of weight functions give known definitions of symbol classes and pseudodifferential operators. For instance, in the limit case, that we do not consider here, of $\omega(t)=\log(1+t)$, we have  $\Sch_{\omega}(\R^{d})=\Sch(\R^{d})$. In this case, with a similar argument to the one of \cite[Example 2.11 (1)]{FGJ2005pseudo}, by \eqref{EqOmegaCorchetex}, we obtain that $a\in \GA_{\rho}^{m,\omega}$ if and only if for all $\alpha, \beta,\gamma\in\N_{0}^{d}$ there is $C>0$ such that
$$
\big| D^{\alpha}_{x} D^{\beta}_{y} D^{\gamma}_{\xi} a(x,y,\xi) \big| \leq C\langle x-y\rangle^{\rho|\alpha+\beta+\gamma|}\langle (x,y,\xi)\rangle^{m-\rho|\alpha+\beta+\gamma|}, 
$$
for all $x,y,\xi\in\R^{d}$. This characterization gives precisely \cite[Definition 23.3]{Shu2001pseudo} for $m'=0$. 

In the same way, using \cite[Example 2.11 (2)]{FGJ2005pseudo}, if $\omega(t)=t^{d}$ for $0<d<1$ is a Gevrey weight function then $a\in \GA_{\rho}^{m,\omega}$ if and only if for every $\lambda>0$ there is $C>0$ such that
\[ \big| D^{\alpha}_{x} D^{\beta}_{y} D^{\gamma}_{\xi} a(x,y,\xi) \big| \leq C \frac{\langle x-y \rangle^{\rho|\alpha + \beta + \gamma|}}{\langle(x,y,\xi)\rangle^{\rho|\alpha + \beta + \gamma|}} (\alpha!\beta!\gamma!)^{\rho/d}\lambda^{|\alpha+\beta+\gamma|}  e^{m\langle (x,y,\xi)\rangle^{d}}, \] for all $(x,y,\xi) \in \mathbb{R}^{3d}$ and $(\alpha,\beta,\gamma) \in \mathbb{N}_{0}^{3d}$. This definition of amplitude could be compared with \cite[Definition 2.1]{cap}, which is the corresponding definition for the Roumieu case. 

Finally, it is worth to mention also 
that in the case when the weight function satisfies \cite[Corollary 16 (3)]{BMM}, the classes of ultradifferentiable functions defined by weights (as in \cite{BMT}) and the ones defined by sequences (as in \cite{K}) coincide. In this situation, the definition given by Prangoski for the Beurling case in \cite[Definition 1]{Pran2013pseudo} is expected to be the same as our Definition~\ref{amplitude}. But, if the weight sequence $(M_{p})_{p}$ satisfies only condition (M2) of Komatsu, as is assumed by Prangoski~\cite{Pran2013pseudo}, our classes of amplitudes could differ in general from the ones given by Prangoski (see \cite[Example 17]{BMM}). Hence, we are treating, even only in the Beurling setting, different cases.
\vskip.3\baselineskip
(b) Let $\sigma$ be a weight function and let $\omega$ be another weight function satisfying $\omega\big(t^{\frac{1+\rho}{\rho}}\big) = O(\sigma(t))$ as $t \to \infty$, where $0 < \rho \leq 1$. If $a(x,\xi) \in \mathcal{S}_{\sigma}\big(\mathbb{R}^{2d}\big)$, then $a(x,\xi) \in\displaystyle \cap_{m\in\R}\textstyle\GS^{m,\omega}_{\rho}$. It is enough to prove it for $m<0$.
Indeed, for every $\lambda,m'>0$, there exists $C_{\lambda,m'}>0$ such that
\begin{equation*}
|D^{\alpha}_x D^{\beta}_{\xi} a(x,\xi)| \leq C_{\lambda,m'} e^{(\lambda+m') \varphi^{\ast}_{\sigma}\big(\frac{|\alpha+\beta|}{\lambda+m'}\big)} e^{(-\lambda-m')\sigma(\langle(x,\xi)\rangle)}\leq C_{\lambda,m'} e^{\lambda \varphi^{\ast}_{\sigma}\big(\frac{|\alpha+\beta|}{\lambda}\big)} e^{(-\lambda-m')\sigma(\langle(x,\xi)\rangle)},
\end{equation*} 
for each $(x,\xi) \in \mathbb{R}^{2d}$ and $(\alpha,\beta) \in \mathbb{N}_0^{2d}$ (in the last inequality we use that $x\mapsto \varphi^{*}(x)/x$ is increasing).
By assumption and Lemma~\ref{Lemma158TesisDavid}\,(2), there is $C>0$ such that for all $\lambda>0$ and $j \in \mathbb{N}_0$,
\begin{equation}\label{Eq2}
\lambda \varphi^{\ast}_{\sigma}\big(\frac{j}{\lambda}\big) \leq \lambda+ \lambda C\frac{\rho}{1+\rho}\varphi^{\ast}_{\omega}\big(\frac{j}{\lambda C}\big).
\end{equation} 
By \eqref{EqLemmaTechnical1}, we have
\begin{eqnarray*}
e^{\lambda \varphi^{\ast}_{\sigma}\big(\frac{|\alpha+\beta|}{\lambda}\big)} e^{-(\lambda+m') \sigma(\langle(x,\xi)\rangle)} &\leq& e^{\lambda(1+\rho) \varphi^{\ast}_{\sigma}\big(\frac{|\alpha+\beta|}{\lambda}\big)}  e^{-\lambda \rho \varphi^{\ast}_{\sigma}\big(\frac{|\alpha+\beta|}{\lambda}\big)} e^{-\lambda \rho \sigma(\langle(x,\xi)\rangle)} e^{-m' \sigma(\langle(x,\xi)\rangle)} \\
&\leq& e^{\lambda(1+\rho) \varphi^{\ast}_{\sigma}\big(\frac{|\alpha+\beta|}{\lambda}\big)} \langle(x,\xi)\rangle^{-\rho|\alpha+\beta|}e^{-m' \sigma(\langle(x,\xi)\rangle)}.
\end{eqnarray*} Now, formula~\eqref{Eq2} shows that $a \in \GS^{m,\omega}_{\rho}$ for $m=-m'/2<0$. 

On the other hand, it is easy to see also that $\bigcap_{m\in\R}\textstyle\GS^{m,\omega}_{\rho}\subseteq \mathcal{S}_{\omega}\big(\mathbb{R}^{2d}\big)$. So, we have 
$$
\mathcal{S}_{\sigma}\big(\mathbb{R}^{2d}\big)\subseteq \displaystyle \bigcap_{m\in\R}\textstyle\GS^{m,\omega}_{\rho}\subseteq \mathcal{S}_{\omega}\big(\mathbb{R}^{2d}\big),
$$
for every pair of weights $\omega$ and $\sigma$ that satisfies the relation at the beginning of this example.

We observe that for weights of the type $\omega(t)=\log^{s}(1+t)$, with $s\ge 1$ (remember that here we do not consider the limit case $s=1$), we have $\omega\big(t^{(1+\rho)/\rho}\big) = O(\omega(t))$ as $t \to \infty$ and, hence, in this particular case, we obtain
$$
\mathcal{S}_{\omega}\big(\mathbb{R}^{2d}\big)= \displaystyle \bigcap_{m\in\R}\textstyle\GS^{m,\omega}_{\rho}.
$$
\vskip.3\baselineskip
(c) We consider the differential operator $P(x,D) = \sum_{|\gamma| \leq m} a_{\gamma}(x) D^{\gamma},$ where $a_{\gamma}\in \Sch_{\sigma}(\R^{d})$ and $\sigma(t)=o(t)$ as $t$ tends to infinity. If $\omega$ is another weight function such that $\omega\big(t^{(1+\rho)/\rho}\big)=O(\sigma(t))$ as $t$ tends to infinity, by (b) it is easy to see that the corresponding symbol $p(x,\xi)=(2\pi)^{-d}\sum_{|\gamma| \leq m} a_{\gamma}(x) \xi^{\gamma}\in\GS^{m,\omega}_{\rho}$ is of {\em finite order}, i.e, we have polynomial growth in all the variables instead of exponential growth.

On the other hand, a linear partial differential operator with polynomial coefficients defines a global symbol of finite order in $\GS^{m,\omega}_{\rho}$ provided $\omega(t^{1/\rho})=o(t)$ as $t$ tends to infinity. 
\vskip.3\baselineskip
(d) Following \cite[Example 2.11 (5)]{FGJ2005pseudo}, we can consider ultradifferential operators with variable coefficients and infinite order $G(x,D):=\sum_{\alpha\in\N^{d}_{0}}a_{\alpha}(x)D^{\alpha}$ with  $a_{\alpha} \in C^{\infty}(\mathbb{R}^d)$  satisfying the following condition: there exists $m \in \mathbb{N}_0$ such that for all $\lambda>0$, there is $C_{\lambda}>0$ with \[ |D^{\beta} a_{\alpha}(x)| \leq C_{\lambda} e^{\lambda \rho \varphi^{\ast}\big(\frac{|\beta|}{\lambda}\big)} e^{-m\rho \varphi^{\ast}\big(\frac{|\beta|}{m}\big)} e^{-m \varphi^{\ast}\big(\frac{|\alpha|}{m}\big)} e^{m\omega(x)}, \] for each $\alpha,\beta \in \mathbb{N}_0^d$ and $x\in\R^{d}$.

It is not difficult to show that its corresponding symbol $p(x,\xi) := (2\pi)^{-d} \sum_{\alpha} a_{\alpha}(x) \xi^{\alpha}$ is a global symbol in $\GS_{\rho}^{k,\omega}$ for some $k\ge m$. 

In particular, by \eqref{EqEstimationAAlpha}, the ultradifferential operators $G(D)$ with constant coefficients defined in Section~\ref{SectionBackground} are pseudodifferential operators with symbol $G\in \GS_{\rho}^{k,\omega}$ for some $k>0$.

}\end{exam}

\section{Symbolic calculus}\label{section-calculus}

In order to compose two pseudodifferential operators, we need to develop a symbolic calculus in this setting. We follow the lines of \cite{FGJ2005pseudo}.
\begin{defin}\label{DefFormalSums}
We define $\FGS_{\rho}^{m,\omega}$ to be the set of all formal sums $\sum_{j \in \mathbb{N}_0} a_j(x,\xi)$ such that $a_j(x,\xi) \in C^{\infty}\big(\mathbb{R}^{2d}\big)$ and there is $R\geq 1$ such that for every $n>0$, there exists $D_n>0$ with
\begin{equation}\label{EqFormalSum}
\big|D^{\alpha}_x D^{\beta}_{\xi} a_j(x,\xi)\big| \leq D_n \Big( \frac1{\langle(x,\xi)\rangle} \Big)^{\rho(|\alpha+\beta|+j)} e^{n\rho \varphi^{\ast}\big(\frac{|\alpha+\beta|+j}{n} \big)} e^{m\omega(x)} e^{m\omega(\xi)},
\end{equation} 
for each $j \in \mathbb{N}_0$, $(\alpha,\beta) \in \mathbb{N}_{0}^{2d}$ and $\log\big(\frac{\langle(x,\xi)\rangle}{R}\big) \geq \frac{n}{j} \varphi^{\ast}\big(\frac{j}{n} \big)$.
\end{defin}
We can assume that $a_0(x,\xi)$ satisfies formula~\eqref{EqFormalSum} when $\log\big(\frac{\langle(x,\xi)\rangle}{R}\big) \geq 0$, i.e., when $\langle (x,\xi)\rangle \ge R$.

Let $a$ be a symbol in $\GS_{\rho}^{m,\omega}$ and set $a_0:=a$ and $a_j=0$ for $j \neq 0$. Then, we can regard  $a$ as the formal sum $\sum a_j$.

\begin{defin}\label{DefEqFormalSums}
Two formal sums $\sum a_j$ and $\sum b_j$ in $\FGS_{\rho}^{m,\omega}$ are said to be equivalent, which is denoted by $\sum a_j \sim \sum b_j$, if there is $R\geq 1$ such that for each natural number $n$, there exist $D_n>0$, $N_n \in \mathbb{N}$ with
\begin{equation}\label{EqFormalSumEquivalent}
\Big|D^{\alpha}_x D^{\beta}_{\xi} \sum_{j < N}(a_j - b_j)\Big| \leq D_n \Big( \frac1{\langle(x,\xi)\rangle} \Big)^{\rho(|\alpha+\beta|+N)} e^{n\rho \varphi^{\ast}\big(\frac{|\alpha+\beta|+N}{n}\big)} e^{m\omega(x)} e^{m\omega(\xi)},
\end{equation} for every $N \geq N_n$, $(\alpha,\beta) \in \mathbb{N}_{0}^{2d}$ and $\log\big(\frac{\langle(x,\xi)\rangle}{R}\big) \geq \frac{n}{N} \varphi^{\ast}\big(\frac{N}{n} \big)$.
\end{defin}

We understand that a symbol $a\in\GS^{m,\omega}_{\rho}$ regarded as a formal sum satisfies $a\sim 0$ when $\big|D^{\alpha}_x D^{\beta}_{\xi} a(x,\xi)\big|$ is estimated by the right-hand side of~\eqref{EqFormalSumEquivalent} for every $N \geq N_n$, $(\alpha,\beta) \in \mathbb{N}_0^{2d}$ and $\log\big(\frac{\langle(x,\xi)\rangle}{R}\big) \geq \frac{n}{N}\varphi^{\ast}\big(\frac{N}{n}\big)$.
The following proposition gives a sufficient condition for a pseudodifferential operator to be $\omega$-regularizing in terms of formal sums (see  Definition~\ref{regu}):

\begin{prop}\label{regu-equiv0}
If $A$ is a pseudodifferential operator defined by a symbol $a(x,\xi)$ which is equivalent to zero, then $A$ is an $\omega$-regularizing operator.
\end{prop}
\begin{proof}
It is enough to show that $a \in \mathcal{S}_{\omega}(\R^{d})$, because  \cite[Proposition 1.2.1]{Nic_Rod2010global} states that operators with symbols in $\mathcal{S}_{\omega}(\R^{d})$ correspond to kernels in $\mathcal{S}_{\omega}(\R^{2d})$ and, by Proposition~\ref{PropExampleExtensionPseudodifferentialOperator}, those operators are $\omega$-regularizing. Since $a \sim 0$, there is $R\geq 1$ such that for every $n>0$, there exist $C_n>0$, $N_n \in \mathbb{N}$ with
\[ |D^{\alpha}_x D^{\beta}_{\xi} a(x,\xi)| \leq C_{8n} \Big( \frac1{\langle(x,\xi)\rangle} \Big)^{\rho(|\alpha+\beta|+N)} e^{8n\rho \varphi^{\ast}\big(\frac{|\alpha+\beta|+N}{8n}\big)} e^{m\omega(x)} e^{m\omega(\xi)}, \] for all $N\geq N_{8n}$, $\log\big(\frac{\langle(x,\xi)\rangle}{R}\big) \geq \frac{8n}{N} \varphi^{\ast}\big(\frac{N}{8n}\big)$, $(\alpha,\beta) \in \mathbb{N}_{0}^{2d}$. We take $0 < \varepsilon < 1$ and $l \in \mathbb{N}$ so that $\omega\big(\frac{t}{R}\big) \geq \varepsilon \omega(t) - \frac1{\varepsilon}$ and $\log(\langle(x,\xi)\rangle) \leq l\omega(x) + l\omega(\xi)$. Observe that there exists $N\geq N_{8n}$ depending on $x,\xi$ and $R$ such that
\[
\left( \frac{8n}{N} \varphi^{\ast}\Big(\frac{N}{8n}\Big) \leq \right) \frac{2n}{N} \varphi^{\ast}\Big(\frac{N}{2n}\Big) \leq \log\Big(\frac{\langle(x,\xi)\rangle}{R}\Big) \leq \frac{2n}{N+1} \varphi^{\ast}\Big(\frac{N+1}{2n}\Big).
\] Now, by Lemma~\ref{LemmaTechnical2},
\begin{eqnarray*}
\Big[ \Big(\frac{\langle(x,\xi)\rangle}{R}\Big)^{-N} e^{4n \varphi^{\ast}\big(\frac{N}{4n}\big)} \Big]^{\rho} &\leq& \Big[ e^{-2n \omega\big(\frac{\langle(x,\xi)\rangle}{R}\big)} e^{\log\big(\frac{\langle(x,\xi)\rangle}{R}\big)} \Big]^{\rho} \leq \Big[ e^{-2n\varepsilon \omega(\langle(x,\xi)\rangle) + \frac{2n}{\varepsilon} + l\omega(x) + l\omega(\xi)} \Big]^{\rho} \\
&\leq& e^{-n \rho \varepsilon \omega(x) -n \rho \varepsilon \omega(\xi) + \frac{2n\rho}{\varepsilon} + \rho l\omega(x) + \rho l\omega(\xi)}.
\end{eqnarray*} Therefore we obtain, since $R$ and $\langle(x,\xi)\rangle$ are greater than or equal to $1$, by the convexity of $\varphi^{*}$,
\begin{eqnarray*}
\big|D^{\alpha}_x D^{\beta}_{\xi} a(x,\xi)\big| &\leq& C_{8n} \langle(x,\xi)\rangle^{-\rho|\alpha+\beta|} R^{-\rho N} \left(\frac{\langle(x,\xi)\rangle}{R}\right)^{-\rho N} e^{4n\rho \varphi^{\ast}\big(\frac{N}{4n}\big)} e^{4n \rho \varphi^{\ast}\big(\frac{|\alpha+\beta|}{4n}\big)} e^{m\omega(x)} e^{m\omega(\xi)} \\
&\leq& C_{8n} e^{\frac{2n\rho}{\varepsilon}} e^{4n \varphi^{\ast}\big(\frac{|\alpha+\beta|}{4n}\big)} e^{(m+\rho l-n\rho\varepsilon)\omega(x)} e^{(m+\rho l-n\rho\varepsilon)\omega(\xi)}.
\end{eqnarray*} 
Now, it suffices to select $n$ large enough.
\end{proof}%

We can also obtain  the opposite of Proposition~\ref{regu-equiv0} for weight functions of the type $\omega(t)=\log^{s}(1+t)$ for $s\ge 1.$ Despite we do not consider in this paper $s=1$, the same argument in this case works, too. We need the following lemma, which holds for any weight function $\omega$.

\begin{lema}\label{regu-reciproco-lema}
Suppose that $a\in\bigcap_{m\in\R} \GS^{m,\omega}_{\rho}$. Then $a\sim0$ in $\FGS^{m,\omega}_{\rho}$ for all $m\in\R$.
\end{lema}

\begin{proof}
First, we observe that there is $C>0$, which only depends on $\omega$, such that 
\begin{equation}\label{cota-omega}
\omega(\langle (x,\xi)\rangle)\le C+C\omega(x)+C\omega(\xi),
\end{equation}
for all $x,\xi\in\R^{d}$. 

Now, we fix $m\in\R.$ By assumption, for all $n\in\N$ there is $E_{n}>0$ (which also depends on $m$) such that 
\begin{eqnarray*}
\lefteqn{\big|D_{x}^{\alpha}D_{\xi}^{\beta}a(x,\xi)\big|\le E_{n}\frac{e^{\rho n\varphi^{*}\big(\frac{|\alpha+\beta|}{n}\big)}}{\langle (x,\xi)\rangle^{\rho|\alpha+\beta|}}e^{-nC(\omega(x)+\omega(\xi))}e^{m(\omega(x)+\omega(\xi))}}\nonumber\\
&&=  E_{n}\frac{e^{\rho n\varphi^{*}\big(\frac{|\alpha+\beta|}{n}\big)}}{\langle (x,\xi)\rangle^{\rho|\alpha+\beta|+\rho N}}\langle (x,\xi)\rangle^{\rho N}e^{-nC(\omega(x)+\omega(\xi))}e^{m(\omega(x)+\omega(\xi))},\nonumber 
\end{eqnarray*}
for all $x,\xi\in\R^{d}$, $\alpha,\beta\in\N_{0}^{d}$ and $N\in\N$. By \eqref{cota-omega}, we have 
\begin{eqnarray*}
-nC\omega(x)-nC\omega(\xi)\le -n\omega(\langle (x,\xi)\rangle)+nC
\le -n\rho\omega(\langle (x,\xi)\rangle)+nC. 
\end{eqnarray*}
Moreover, by \eqref{EqLemmaTechnical1}, 
$$
\langle (x,\xi)\rangle^{N}e^{-n\omega(\langle (x,\xi)\rangle)}\le e^{n\varphi^{*}\big(\frac{N}{n}\big)}.
$$
Therefore, we obtain that for each $n\in\N$ there is $C_{n}>0$ such that 
$$
\big|D_{x}^{\alpha}D_{\xi}^{\beta}a(x,\xi)\big|\le C_{n}\frac{e^{\rho n\varphi^{*}\big(\frac{|\alpha+\beta|+N}{n}\big)}}{\langle (x,\xi)\rangle^{\rho|\alpha+\beta|+\rho N}}e^{m(\omega(x)+\omega(\xi))},
$$
for all $x,\xi\in\R^{d}$, $\alpha,\beta\in\N_{0}^{d}$ and $N\in\N$. Since the argument does not depend on $m\in\R$, we have $a\sim0$ in $\FGS^{m,\omega}_{\rho}$ for each $m\in\R.$
\end{proof}
\begin{prop}\label{equiv0-reciproco}
Let $\omega(t)=\log^{s}(1+t)$, for $s\ge 1.$ If $A$ is an $\omega$-regularizing operator with symbol $a$, we have $a\sim0$ in $\FGS^{m,\omega}_{\rho}$ for all $m\in\R.$ 
\end{prop}

\begin{proof}
Since $A$ is $\omega$-regularizing, the symbol $a\in\Sch_{\omega}(\R^{2d})$ by Propostion~\ref{PropExampleExtensionPseudodifferentialOperator} and \cite[Proposition 1.2.1]{Nic_Rod2010global}. By the argument given in Example~\ref{examples}\,(b) for weights $\omega(t)=\log^{s}(1+t)$, for $s\ge 1,$ we obtain $a\in\bigcap_{m\in\R} \GS^{m,\omega}_{\rho}$. Hence, Lemma~\ref{regu-reciproco-lema} gives the conclusion.
\end{proof}

Now, we construct a symbol from a formal sum, and to do so we need some kind of partition of unity. Here, we cannot use the estimates as in \cite[Lemma 3.6]{FGJ2005pseudo} for some technical difficulties, but we consider the usual estimates for ultradifferentiable functions instead. This is due to the fact that our symbols are defined in the whole $\R^{d}$ for all the variables. However, we observe that this consideration is not so restrictive (cf. \cite[Remark 1.7~(1)]{FGJ2005pseudo}).

We consider $\Phi(x,\xi) \in \mathcal{D}_{(\sigma)}\big(\mathbb{R}^{2d}\big)$, where $\sigma$ and $\omega$ are weight functions which satisfy $\omega(t^{1/\rho})=O(\sigma(t))$ when $t\to+\infty$ (Lemma~\ref{Lemma158TesisDavid}(2)) and, in addition,
\[ |\Phi(x,\xi)| \leq 1, \qquad \Phi(x,\xi)=1 \ \text{if} \ |(x,\xi)| \leq 2, \qquad \Phi(x,\xi)=0 \ \text{if} \ |(x,\xi)| \geq 3. \]
Let $(j_n)_{n}$ be an increasing sequence of natural numbers such that $j_n/n \to \infty$ as $n$ tends to infinity. For each $j_n \leq j < j_{n+1}$, we set
\begin{equation}
\label{psi-function}\Psi_{j,n}(x,\xi) := 1 - \Phi\Big(\frac{(x,\xi)}{A_{n,j}}\Big), \qquad A_{n,j}=Re^{\frac{n}{j}\varphi_{\omega}^{\ast}(\frac{j}{n})}, \end{equation} where $R\geq 1$ is the constant which appears in Definition~\ref{DefFormalSums}. It is clear that $A_{n,j}^{\rho} \leq A_{n,j}$. We observe that $(x,\xi) \in \supp \Psi_{j,n}$ implies $\big| \frac{(x,\xi)}{A_{n,j}} \big| > 2$ and so
\begin{equation}\label{EqInequalityPsi1}
\langle(x,\xi)\rangle > 2A_{n,j}.
\end{equation}
Since $\Phi \in\mathcal{D}_{(\sigma)}\big(\mathbb{R}^{2d}\big)$, for each $k \in \mathbb{N}$ there is a constant $C_{k}>0$ such that
$|D^{\alpha}_x D^{\beta}_{\xi} \Phi(x,\xi)| \leq C_{k} e^{k\varphi^{\ast}_{\sigma}\big(\frac{|\alpha+\beta|}{k}\big)}$. Now, by Lemma~\ref{Lemma158TesisDavid}(2), for all $k\in\N$ there is $C_{k}>0$ with
\begin{equation}\label{EqEstimateDerivativesPsijn1}
|D^{\alpha}_x D^{\beta}_{\xi} \Psi_{j,n}(x,\xi)| = \Big|D^{\alpha}_x D^{\beta}_{\xi} \Phi\Big( \frac{(x,\xi)}{A_{n,j}}\Big)\Big| A_{n,j}^{-|\alpha+\beta|} \leq C_k e^{k \rho \varphi_{\omega}^{\ast}\big(\frac{|\alpha+\beta|}{k}\big)} A_{n,j}^{-\rho|\alpha+\beta|},
\end{equation} for each $(\alpha,\beta) \in \mathbb{N}_{0}^{2d}$ and all $k\in \N$. If additionally we assume that $(x,\xi)$ is in the support of any derivative of $\Psi_{j,n}(x,\xi)$, we have $2 \leq \big| \frac{(x,\xi)}{A_{n,j}} \big| \leq 3$. This implies
\begin{equation}\label{EqInequalityPsi2}
2A_{n,j} \leq \langle(x,\xi)\rangle \leq \sqrt{10}A_{n,j}.
\end{equation} 
We obtain, from~\eqref{EqEstimateDerivativesPsijn1}, that for all $k \in \mathbb{N}$, 
\begin{equation}\label{EqEstimateDerivativesPsijn2}
|D^{\alpha}_x D^{\beta}_{\xi} \Psi_{j,n}(x,\xi)| \leq C_k \Big( \frac{\sqrt{10}}{\langle(x,\xi)\rangle} \Big)^{\rho|\alpha+\beta|} e^{k\rho \varphi^{\ast}\big(\frac{|\alpha+\beta|}{k}\big)}.
\end{equation} 
Hence $\Psi_{j,n}\in \GS^{0,\omega}_{\rho}$ (here, we apply Lemma~\ref{lemmafistrella}~(1) to get rid of the constant $(\sqrt{10})^{\rho|\alpha+\beta|}$).
The proof of the next results follow the lines of the one of \cite[Theorem 3.7]{FGJ2005pseudo}:
\begin{theo}\label{TheoEquivalentSymbol}
Let $\sum a_j$ be a formal sum in $\FGS^{m,\omega}_{\rho}$. Then there exists a global symbol $a \in \GS^{m,\omega}_{\rho}$ such that $a \sim \sum a_j$.
\end{theo}
\begin{proof}
We consider the functions $\Psi_{j,n}$ defined in \eqref{psi-function}. Since $\Psi_{j,n} \neq 0$ implies that formula~\eqref{EqInequalityPsi1} holds, we also have
\begin{equation}\label{EqTheoExistenceSymbol1}
\langle(x,\xi)\rangle^{-\rho j} e^{n \rho \varphi^{\ast}\big( \frac{j}{n} \big)} < (2R)^{-\rho j}.
\end{equation} If we suppose that $(x,\xi)$ belongs to the support of any derivative of $\Psi_{j,n}$, then formula~\eqref{EqInequalityPsi2} is satisfied. In particular, we have
\[ \log\Big( \frac{\langle(x,\xi)\rangle}{\sqrt{10}R} \Big) \leq \frac{n}{j}\varphi^{\ast}\big(\frac{j}{n}\big). \]

It is not difficult to see that, by formula~\eqref{EqTheoExistenceSymbol1}, 
\begin{eqnarray}
\lefteqn{\big|D^{\alpha}_x D^{\beta}_{\xi} \big( a_j(x,\xi) \Psi_{j,n}(x,\xi) \big)\big|} \nonumber \\
&& \nonumber \le D_n \langle(x,\xi)\rangle^{-\rho(|\alpha+\beta|+j)} e^{2n\rho \varphi^{\ast}\big(\frac{|\alpha+\beta|+j}{2n}\big)} e^{m\omega(x)} e^{m\omega(\xi)} \\ \label{Eq2jn}
&& \leq D_n \langle(x,\xi)\rangle^{-\rho|\alpha+\beta|} e^{n \rho \varphi^{\ast}\big(\frac{|\alpha+\beta|}{n}\big)} e^{m\omega(x)} e^{m\omega(\xi)} (2R)^{-\rho j},
\end{eqnarray} 
for some constant $D_n > 0$, for all $j \in \mathbb{N}_0$, $(\alpha,\beta)\in \mathbb{N}_0^{2d}$ and  $\log\big(\frac{\langle(x,\xi)\rangle}{2R}\big)\geq \frac{n}{j}\varphi^{\ast}\big(\frac{j}{n}\big)$. This shows that $a_j(x,\xi) \Psi_{j,n}(x,\xi)$ is a global symbol, since $\log\big(\frac{\langle(x,\xi)\rangle}{2R}\big)\le \frac{n}{j}\varphi^{\ast}\big(\frac{j}{n}\big)$ implies that $\Psi_{j,n}(x,\xi)=0$, by \eqref{EqInequalityPsi1}.

We observe that $\sum_{j=1}^{\infty} (2R)^{-\rho j}$ is convergent, because $R\geq1$. Let $(j_n)_{n}$ be the sequence which defines the functions $\Psi_{j,n}$. By induction, we can take the elements of $(j_n)_{n}$ so that $j_1:=1$, $j_n < j_{n+1}$, $\frac{j_n}{n} \to \infty$ and
\[ D_{n+1} \sum_{j=j_{n+1}}^{\infty} \frac1{(2R)^{\rho j}} \leq \frac{D_n}{2} \sum_{j=j_n}^{j_{n+1}-1} \frac1{(2R)^{\rho j}}. \] Then it is easy to check that
\[ \overline{D}_n:= D_n \sum_{j=j_n}^{j_{n+1}-1} \frac1{(2R)^{\rho j}} \] satisfies that $\overline{D}_{n+1} \leq \frac{\overline{D}_n}{2}$. 

On the other hand, it is not difficult to see that
\[ a(x,\xi) = a_0(x,\xi) + \sum_{n=1}^{\infty} \sum_{j=j_n}^{j_{n+1}-1} \Psi_{j,n}(x,\xi) a_j(x,\xi) \] is a global symbol in $\GS^{m,\omega}_{\rho}$.

Now, we claim that $a \sim \sum a_j$. Assume $\log\big(\frac{\langle(x,\xi)\rangle}{\sqrt{10}R}\big) \geq \frac{n}{N} \varphi^{\ast}\big(\frac{N}{n}\big)$. We consider only the case $N\ge n j_{n}$ (which is coherent with Defintion~\ref{DefEqFormalSums}). For all $j\in \N$ there is $k\in\N$ with $j_{k}\le j<j_{k+1}$.  If $k<n$, we have $j\le j_{n}(<N)$ and therefore
$$
\log\big(\frac{\langle(x,\xi)\rangle}{\sqrt{10}R}\big) \geq \frac{n}{N} \varphi^{\ast}\big(\frac{N}{n}\big)\ge \frac{1}{j_{n}} \varphi^{\ast}(j_{n})\ge \frac{k}{j} \varphi^{\ast}\big(\frac{j}{k}\big).
$$
In this case $\Psi_{j,k}\equiv 1$. If $k\ge n$ and $N>j$ we have
$$
\log\big(\frac{\langle(x,\xi)\rangle}{\sqrt{10}R}\big) \geq \frac{n}{N} \varphi^{\ast}\big(\frac{N}{n}\big)\ge \frac{k}{j} \varphi^{\ast}\big(\frac{j}{k}\big),
$$
and also $\Psi_{j,k}\equiv 1$. Hence, we only have to analyse the case when $j\geq N$ and $k \geq n$. 

So, we are looking for an estimate for
$\big|D^{\alpha}_x D^{\beta}_{\xi} \Psi_{j,k}(x,\xi) a(x,\xi)\big|$ with $j \geq N$ and $k \geq n$. We assume that $\log\big(\frac{\langle(x,\xi)\rangle}{2R}\big) \geq \frac{k}{j} \varphi^{\ast}\big(\frac{j}{k}\big)$ (since, otherwise, $\Psi_{j,k}=0$). Now, we have, by the convexity of $\varphi^{*}$ and Leibniz's rule, 
\begin{eqnarray}
\lefteqn{\nonumber \big|D^{\alpha}_x D^{\beta}_{\xi} \big(\Psi_{j,k}(x,\xi) a_j(x,\xi) \big)\big| } \\
&& \nonumber \leq D_k \langle(x,\xi)\rangle^{-\rho(|\alpha+\beta|+N)} e^{k\rho \varphi^{\ast}\big(\frac{|\alpha+\beta|+N}{k}\big)}  \label{Eq3}
 \langle(x,\xi)\rangle^{-\rho(j-N)} e^{k\rho\varphi^{\ast}\big(\frac{j-N}{k}\big)} e^{m\omega(x)} e^{m\omega(\xi)}. 
\end{eqnarray} 
We obtain 
\[ \Big( \langle(x,\xi)\rangle^{-(j-N)} e^{kL \varphi^{\ast}\big(\frac{j-N}{kL}\big)} \leq \Big)\langle(x,\xi)\rangle^{-(j-N)} e^{k \varphi^{\ast}\big(\frac{j-N}{k}\big)} \leq (2R)^{-(j-N)}, \] and thus, for its $\rho$-power also. Therefore, for $k\ge n$, $j\ge N$ and the constants $(D_{k})_{k\ge n}$ as in \eqref{Eq2jn} we have
\begin{eqnarray}
\lefteqn{\big|D^{\alpha}_x D^{\beta}_{\xi} \big(\Psi_{j,k}(x,\xi) a_j(x,\xi) \big)\big|}\nonumber\\&&\ \ \le  D_k\label{Eq4}
 \langle(x,\xi)\rangle^{-\rho(|\alpha+\beta|+N)} e^{k\rho \varphi^{\ast}\big(\frac{|\alpha+\beta|+N}{k}\big)} (2R)^{-\rho(j-N)} e^{m\omega(x)} e^{m\omega(\xi)}.
\end{eqnarray} 
Since $k\geq n$ and $j \geq N$, we get
\begin{eqnarray}
 \Big| D^{\alpha}_x D^{\beta}_{\xi} \Big( a - \sum_{j<N} a_j\Big)\Big|  \label{Eq5}
\le  \sum_{k \geq n} \sum_{\substack{j=j_k\\ j\ge N}}^{j_{k+1}-1} |D^{\alpha}_x D^{\beta}_{\xi} (\Psi_{j,k} a_j)|.
\end{eqnarray} Now, $k\geq n$ also implies $k \rho \varphi^{\ast}\big(\frac{|\alpha+\beta|+N}{k}\big) \leq n \rho \varphi^{\ast}\big(\frac{|\alpha+\beta|+N}{n}\big)$. Therefore, using~\eqref{Eq4}, we can estimate \eqref{Eq5} by
\begin{eqnarray*}
\lefteqn{ (2R)^{\rho N} \langle(x,\xi)\rangle^{-\rho(|\alpha+\beta|+N)} e^{m\omega(x)} e^{m\omega(\xi)} \sum_{k\geq n} D_k e^{k\rho \varphi^{\ast}\big(\frac{|\alpha+\beta|+N}{k}\big)} \sum_{\substack{j=j_k\\ j\ge N}}^{j_{k+1}-1} (2R)^{-\rho j}} \\
&& \leq  (2R)^{\rho N} \langle(x,\xi)\rangle^{-\rho(|\alpha+\beta|+N)} e^{n\rho \varphi^{\ast}\big(\frac{|\alpha+\beta|+N}{n}\big)} e^{m\omega(x)} e^{m\omega(\xi)} \sum_{k\geq n} \overline{D}_k,
\end{eqnarray*} where $\sum_{k\geq n} \overline{D}_k$ is a constant depending on $n$, which finishes the proof.
\end{proof}

From now on, we assume that $\frac{n}{j}\varphi^{\ast}\big(\frac{j}{n}\big) \geq n$ for every $j \geq j_n$. For every $n \in \mathbb{N}$, we define, for $j_n \leq j < j_{n+1}$,
\begin{equation}\label{not18}
\varphi_j := \Psi_{j,n}, \qquad \qquad \varphi_0 = 1. 
\end{equation} 
A simple computation gives $A_{n,j} \leq A_{n,j+1}$. Since $\varphi_j$, $\varphi_{j+1} \in \mathcal{D}_{(\sigma)}\big(\mathbb{R}^{2d}\big)$, we observe that the difference $\varphi_j - \varphi_{j+1}$ belongs to $\mathcal{D}_{(\sigma)}\big(\mathbb{R}^{2d}\big)$.
Therefore, by \eqref{EqEstimateDerivativesPsijn2}, $\varphi_j - \varphi_{j+1}\in \GS^{0,\omega}_{\rho}$.

\begin{lema}\label{LemmaSumOfAmplitudes}
Let $a(x,y,\xi)$ be an amplitude in $\GA^{m,\omega}_{\rho}$, and let $A$ be the corresponding pseudodifferential operator. For each $u \in \mathcal{S}_{\omega}(\R^{d})$,
\[ A(u) = \sum_{j=0}^{\infty} A_j(u) \] in the topology of $\mathcal{S}_{\omega}(\R^{d})$, where $A_j$, $j\geq0$, is the pseudodifferential operator given by the amplitude $(\varphi_j-\varphi_{j+1})(x,\xi)a(x,y,\xi)$.
\end{lema}
\begin{proof}
For $j_n \leq j < j_{n+1}$, it is not difficult to see that $\big( \varphi_j - \varphi_{j+1}\big)(x,\xi) a(x,y,\xi)\in \GA^{m,\omega}_{\rho}$. 
We have, for $u \in \mathcal{S}_{\omega}(\R^{d})$,
\begin{eqnarray*}
\sum_{j=0}^{\infty} A_j(u)= \lim_{N \to \infty} \int_{\mathbb{R}^d} \int_{\mathbb{R}^d} e^{i(x-y)\xi} \big(1-\varphi_{N+1}(x,\xi)\big) a(x,y,\xi) u(y) dy d\xi.
\end{eqnarray*} Now, we observe that, for each $(x,\xi)\in \R^{2d}$,
\[ (1-\varphi_{N+1})(x,\xi) = \Phi\Big(\frac{(x,\xi)}{A_{n,N+1}}\Big), \] where $\Phi\in\D_{(\sigma)}(\R^{2d}) \subset \Sch_{\omega}(\R^{2d})$ is like in formula \eqref{psi-function}. Hence, $\Phi(0,0)=1$. Moreover, $A_{n,N+1}\to \infty$ as $N\to \infty$. Therefore, proceeding as in the proof of Theorem~\ref{pseudotheo} we have
\[ A(u)(x) = \lim_{N \to \infty} \int \Big( \int \big( 1-\varphi_{N+1}\big)(x,\xi) e^{i(x-y)\xi} a(x,y,\xi) u(y) dy \Big) d\xi, \] and the result follows.
\end{proof}

Below, we denote sometimes $\Sch_{\omega}(\R^{d})$ by $\Sch_{\omega}.$

\begin{prop}\label{LemmaPseudoDifferentialOperatorLimit}
Let $\sum_{j=0}^{\infty} p_j(x,\xi)$ be a formal sum in $\FGS_{\rho}^{m,\omega}$ and $(C_n)_{n}$ be the corresponding sequence which appears in~\eqref{EqFormalSum}. Let $(j_n)_{n}$ be a sequence as in Theorem~\ref{TheoEquivalentSymbol} which also satisfies that $\frac{n}{j}\varphi^{\ast}\big(\frac{j}{n}\big) \geq \max\{ n,\log C_n\}$ for $j \geq j_n$, $n\in\N$. We set
\[ p(x,\xi) := \sum_{j=0}^{\infty} \varphi_j(x,\xi) p_j(x,\xi), \] which is a symbol, where $\varphi_{j}$ is the function in \eqref{not18}. Then, its corresponding pseudodifferential operator $P(x,D)$ is the limit in $L\big(\mathcal{S}_{\omega}, \mathcal{S}'_{\omega}\big)$ of the sequence of operators
\[ S_N: \mathcal{S}_{\omega}(\R^{d}) \to \mathcal{S}_{\omega}(\R^{d}),\quad N\in\N, \] where each $S_N$ is a pseudodifferential operator with symbol
\[ \sum_{j=0}^N \big( \varphi_j - \varphi_{j+1} \big)(x,\xi) \Big(\sum_{l=0}^j p_l(x,\xi) \Big),\quad N\in\N. \]
\end{prop}
\begin{proof}
By Theorem~\ref{TheoEquivalentSymbol}, the function $p(x,\xi)$ is a symbol. Moreover, for each $j \in \mathbb{N}_0$, one can show that
\[ \big( \varphi_j - \varphi_{j+1} \big)(x,\xi) \Big(\sum_{l=0}^j p_l(x,\xi) \Big) = \sum_{l=0}^{j} \big( \varphi_j - \varphi_{j+1} \big)(x,\xi) p_l(x,\xi)\] is also a global symbol in $\GS^{m,\omega}_{\rho}$. 
Hence, the function
\[ \sum_{j=0}^{N} \big( \varphi_j - \varphi_{j+1} \big)(x,\xi) \Big( \sum_{l=0}^{j} p_l(x,\xi) \Big) = \sum_{j=0}^{N} \varphi_j(x,\xi) p_j(x,\xi) - \varphi_{N+1}(x,\xi) \sum_{l=0}^{N} p_l(x,\xi) \] is a global symbol in $\GS^{m,\omega}_{\rho}$ since it is a finite sum of global symbols.

 Now, we prove that $S_N \to P$ in $L\big(\mathcal{S}_{\omega}, \mathcal{S}'_{\omega}\big)$ as $N\to+\infty$. Since $\Sch_{\omega}$ is a Fr\'echet-Montel space, it is enough to show that, for any $f,u \in \mathcal{S}_{\omega}$, \[ \langle (S_N-P)f,u \rangle \to 0\ \ \mbox{ as }N\to+\infty.\] The operators $P$ and $S_{N}$, $N=1,2,\ldots$, act continuously from $\mathcal{S}_{\omega}$ into itself. So, we have $(S_N-P)f \in \mathcal{S}_{\omega}$ and
\begin{eqnarray*}\label{S_N}
\lefteqn{\langle(S_N-P)f,u\rangle = \int (S_N-P)f(x) u(x) dx}\\
&&=\int\Big(\int e^{i x\xi}\Big(\sum_{j=0}^{N} \varphi_j(x,\xi) p_j(x,\xi) - \varphi_{N+1}(x,\xi) \sum_{l=0}^{N} p_l(x,\xi)-p(x,\xi)\Big)\hat{f}(\xi)d\xi \Big)u(x)dx\nonumber
\end{eqnarray*}
for each $f,u \in \mathcal{S}_{\omega}$.
%
We will see that for each $f,u \in \mathcal{S}_{\omega}$:
\begin{enumerate}
\item[a)] $\int\Big(\int e^{i x\xi}\Big(\sum_{j=N+1}^{\infty} \varphi_j(x,\xi) p_j(x,\xi) \Big)\hat{f}(\xi)d\xi \Big)u(x)dx\to 0$, and
\item[b)] $\int\Big(\int e^{i x\xi}\Big( \varphi_{N+1}(x,\xi) \sum_{l=0}^{N} p_l(x,\xi)\Big)\hat{f}(\xi)d\xi \Big)u(x)dx\to 0$
\end{enumerate}
when $N\to \infty.$

First, since $f,u \in \mathcal{S}_{\omega}$, there exists a constant $D>0$ depending on $m$ and $L$ (the constant of \eqref{EqOmegaDiff}) such that (Definition~\ref{def3})
\begin{eqnarray*}
|\widehat{f}(\xi)| &\leq& D e^{-(m+L+1)\omega(\xi)}, \\
|u(x)| &\leq& D e^{-(m+L+1)\omega(x)}.
\end{eqnarray*}

Now, when $\varphi_j(x,\xi)\neq 0$ and $j_{n}\le j< j_{n+1}$, we have $\log\big(\frac{\langle(x,\xi)\rangle}{2R}\big) \geq \frac{n}{j}\varphi^{\ast}\big(\frac{j}{n}\big)$, and for the selected sequence $(C_{n})_{n}$, we obtain the estimate
\[ |p_j(x,\xi)| \leq C_n e^{m\omega(x)} e^{m\omega(\xi)} \langle(x,\xi)\rangle^{-\rho j} e^{n\rho \varphi^{\ast}\big(\frac{j}{n}\big)} \leq C_n e^{m\omega(x)} e^{m\omega(\xi)} (2R)^{-\rho j}. \]
Hence (since $|\varphi_{j}(x,\xi)|\le 2$), we have
$$
|u(x) \varphi_{j}(x,\xi) p_{j}(x,\xi) \hat{f}(\xi)|\le 2D^{2} C_{n} (2R)^{-\rho j} e^{-(L+1)(\omega(x)+\omega(\xi))}.
$$
Moreover, we observe that $\omega(x,\xi) \leq L\omega(x) + L\omega(\xi) + L$ (by \eqref{EqOmegaDiff}), and since $\log(t) = o(\omega(t))$ for $t \to \infty$, for $(x,\xi) \in \supp \varphi_j$, we can assume (for $j$ big enough)
\begin{eqnarray*}
\label{eomega}
e^{-L} e^{-L\omega(x)} e^{-L\omega(\xi)} \leq e^{-\omega(x,\xi)}\le \frac{1}{\langle (x,\xi)\rangle} \le \frac1{2R e^{\frac{n}{j} \varphi^{\ast}\big(\frac{j}{n}\big)}}.
\end{eqnarray*}
By these estimates, and taking into account that $\log C_n \leq \frac{n}{j}\varphi^{\ast}\big(\frac{j}{n}\big)$ for $n\in\N$ and $j_{n}\le j<j_{n+1}$, we get for $j_{l}\le N+1<j_{l+1}$,
\begin{eqnarray*}
\sum_{j=N+1}^{\infty} |u(x) \varphi_{j}(x,\xi) p_{j}(x,\xi) \hat{f}(\xi)|\le 2D^{2} e^{L} e^{-(\omega(x)+\omega(\xi))}\sum_{n=l}^{\infty} \sum_{j=j_{n}}^{j_{n+1}-1} \frac{C_{n}}{(2R)^{\rho j}e^{\frac{n}{j}\varphi^{\ast}(\frac{j}{n})}},
\end{eqnarray*}
which proves a) since the integral $\iint e^{-(\omega(x)+\omega(\xi))} d\xi dx$ is convergent.

To see b), given $N$ we take $n$ with $j_{n}\le N+1<j_{n+1}$ and observe that $\varphi_{N+1}(x,\xi)\neq 0$ implies $\log\big(\frac{\langle(x,\xi)\rangle}{2R}\big) \geq \frac{n}{N+1}\varphi^{\ast}\big(\frac{N+1}{n}\big)$. As before, $|\varphi_{j}(x,\xi)|\le 2$ and (for $N$ big enough) $$e^{-\omega(x,\xi)}\le \frac{1}{\langle (x,\xi)\rangle} \le \frac1{2R e^{\frac{n}{N+1} \varphi^{\ast}\big(\frac{N+1}{n}\big)}},$$ so we obtain
\begin{eqnarray*}
\big|u(x)\varphi_{N+1}(x,\xi) \big(\sum_{j=0}^{N} p_{j}(x,\xi)\big)\hat{f}(\xi)\big|&\le& 2D^{2}C_{n} e^{L}e^{-(\omega(x)+\omega(\xi))}e^{-\frac{n}{N+1} \varphi^{\ast}\big(\frac{N+1}{n}\big)}\sum_{j=0}^{N} \frac1{(2R)^{\rho j}}\\
&\le & C e^{-(\omega(x)+\omega(\xi))}e^{-\frac{n}{N+1} \varphi^{\ast}\big(\frac{N+1}{n}\big)},
\end{eqnarray*}
where $C:=2D^{2}C_{n} e^{L} \sum_{j=0}^{\infty} \frac1{(2R)^{\rho j}}$. This concludes the proof, since $j_{n}\le N+1<j_{n+1}$ implies $$\frac{n}{N+1} \varphi^{\ast}\big(\frac{N+1}{n}\big)\ge n.$$
\end{proof}

\subsection{Properties of formal sums}\label{SectionPropertiesFormalSums}
The following results are easy to check:
\begin{exam}\label{ExamFormalSum1}
Let $a(x,y,\xi)$ be an amplitude in $\GA^{m,\omega}_{\rho}$ and let $p_j(x,\xi) := \sum_{|\alpha|=j} \frac1{\alpha!} D^{\alpha}_{\xi} \partial^{\alpha}_y a(x,y,\xi) \left. \right|_{y=x}$. Then the series $\sum_{j=0}^{\infty} p_j(x,\xi)$ is a formal sum in $\FGS^{2m,\omega}_{\rho}$.
\end{exam}

\begin{prop}\label{PropDefinitionTransposition}
Let $\sum p_j \in \FGS^{m,\omega}_{\rho}$ be a formal sum. Then, the sequence $(q_j)_{j}$ given by $q_j(x,\xi):= \sum_{|\alpha|+h=j} \frac1{\alpha!} \partial^{\alpha}_{\xi} D^{\alpha}_x \big( p_h(x,-\xi) \big)$ is a formal sum for each $j \in \mathbb{N}$.
\end{prop}

\begin{defin}
For $\sum p_j \in \FGS^{m,\omega}_{\rho}$, we define $(\sum p_j)^t$ as the formal sum $\sum_j q_j$, where
\[ q_j(x,\xi) := \sum_{|\alpha|+h=j} \frac1{\alpha!} \partial^{\alpha}_{\xi} D^{\alpha}_x\big( p_h(x,-\xi) \big). \] In particular, if $p(x,\xi)$ is a symbol, $p^t(x,\xi)$ denotes the formal sum $\sum_j q_j$ defined by
\[ q_j(x,\xi) := \sum_{|\alpha|=j} \frac1{\alpha!} \partial^{\alpha}_{\xi} D^{\alpha}_x\big( p(x,-\xi) \big).\]
\end{defin}


\begin{prop}\label{PropDefinitionComposition}
Let $\sum p_j \in \FGS^{m_1,\omega}_{\rho}$ and $\sum q_j \in \FGS^{m_2,\omega}_{\rho}$ be two formal sums. The sequence $(r_j)$, defined by $r_j(x,\xi) = \sum_{|\alpha|+k+h=j} \frac1{\alpha!} \partial^{\alpha}_{\xi} p_h(x,\xi) D^{\alpha}_x q_k(x,\xi)$ is a formal sum in $\FGS^{m_1+m_2,\omega}_{\rho}$.
\end{prop}

\begin{defin}\label{DefinCompositionFormalSum}
For $\sum p_j \in \FGS^{m_1,\omega}_{\rho}$, $\sum q_j \in \FGS^{m_2,\omega}_{\rho}$, we define $(\sum p_j) \circ (\sum q_j) = \sum r_j$, where
\[ r_j(x,\xi) = \sum_{|\alpha|+h+k=j} \frac1{\alpha!} \partial^{\alpha}_{\xi} p_h(x,\xi) D^{\alpha}_x q_k(x,\xi). \]
\end{defin}

\begin{prop}\label{PropComposition1}
If $\sum p_j \sim \sum p'_j$ and $\sum q_j \sim \sum q'_j$, then $(\sum p_j) \circ (\sum q_j) \sim (\sum p'_j) \circ (\sum q'_j)$.
\end{prop}

\section{Composition of operators and the transpose operator}\label{section-composition}


First, we study the kernel of a pseudodifferential operator and we show  that it behaves like a $\Sch_{\omega}$-function outside of an arbitrary strip around the diagonal, similarly to the local case; see \cite{FGJ2005pseudo,Nic_Rod2010global}.


\subsection{The behaviour of the kernel of a pseudodifferential operator outside the diagonal}

For any $r>0$, we denote
\[ \Delta_r := \big\{ (x,y) \in \mathbb{R}^{2d} : |x-y| < r\big\}. \]

\begin{lema}\label{LemmaExistenceFunctionEOmega}
Given $r>0$, there exists $\chi \in \mathcal{E}_{(\omega)}\big(\mathbb{R}^{2d}\big)$ such that $0 \leq \chi \leq 1$, $\chi(x,y)=1$ if $(x,y) \in \mathbb{R}^{2d} \setminus \Delta_{r}$ and $\chi(x,y)=0$ if $(x,y) \in \Delta_{\frac{r}{2}}$, which satisfies that for every $\lambda>0$ there exists $C_{\lambda}>0$ with
\[ |D^{\alpha}_x D^{\beta}_y \chi(x,y)| \leq C_{\lambda} e^{\lambda \varphi^{\ast}\big(\frac{|\alpha+\beta|}{\lambda}\big)}, \qquad \alpha,\beta \in \mathbb{N}_0^d, \ x,y \in \mathbb{R}^d. \]
\end{lema}
\begin{proof}
Let $\varphi \in \mathcal{E}_{(\omega)}(\mathbb{R}^{d})$ such that $\varphi(\xi)=0$ if $|\xi| < \frac{r}{2}$, $\varphi(\xi)=1$ if $|\xi| \ge r$, and $0 \leq \varphi \leq 1$. The desired function is $\chi(x,y) = \varphi(x-y)$.
\end{proof}
The next result is crucial for the proof of Theorem~\ref{TheoLarguisimo}. We observe that it is stronger than the ones  given in \cite[Theorem 6.3.3]{Nic_Rod2010global} and \cite[Proposition 5]{Pran2013pseudo}.
\begin{theo}\label{TheoFormalKernel}
Given $r>0$ and an amplitude $a(x,y,\xi) \in \GA^{m,\omega}_{\rho}$, we have  that the \emph{formal} kernel \[ K(x,y) :=  \int_{\mathbb{R}^d} e^{i(x-y)\xi} a(x,y,\xi) d\xi \] satisfies:
\begin{enumerate}
\item $K(x,y) \in C^{\infty}\big(\mathbb{R}^{2d} \setminus \overline{\Delta_r} \big)$,
\item For every $\lambda>0$ there exists $C_{\lambda}>0$ (which depends on $r>0$) such that for all $(x,y) \in \mathbb{R}^{2d} \setminus \Delta_r$ and all $\alpha,\beta \in \mathbb{N}_0^d$, we have
\[ |D^{\alpha}_x D^{\beta}_{y} K(x,y)| \leq C_{\lambda} e^{\lambda \varphi^{\ast}\big(\frac{|\alpha+\beta|}{\lambda}\big)} e^{-\lambda\omega(x)} e^{-\lambda\omega(y)}.\]
\end{enumerate}
\end{theo}
\begin{proof} Let $\sigma$ be a weight function as in Lemma~\ref{Lemma158TesisDavid}(2) with $a=\rho$. We consider $\Psi \in \mathcal{D}_{(\sigma)}\big(\mathbb{R}^{2d}\big)$ such that $\Psi(x,\xi)=1$ if $\langle(x,\xi)\rangle \leq 2$ and
$\Psi(x,\xi)=0$ if $\langle(x,\xi)\rangle \geq 3$. We write
\[ K_n(x,y) = \int_{\mathbb{R}^d} e^{i(x-y)\xi} a(x,y,\xi) \Psi\big(\frac{x}{2^n},\frac{\xi}{2^n}\big) d\xi. \] 

We denote by $A_n$ the operator associated to the kernel $K_n$. By Theorem~\ref{pseudotheo}, it is easy to see  that $K_n \to K$ in $\mathcal{S}'_{\omega}\big(\mathbb{R}^{2d}\big)$. 

Given $(x,y) \in \mathbb{R}^{2d} \setminus \Delta_r$, there is $c_0>0$ independent of $(x,y)\notin\Delta_{r}$ such that $|x-y|_{\infty} \geq c_0$. We can assume that for a given point $(x,y)\notin\Delta_{r}$, $|x-y|_{\infty}=|x_{l}-y_{l}|$ for some $1\le l\le d$.
We will proceed similarly to the proof of \cite[Theorem 2.17]{FGJ2005pseudo}, but here we need to apply a further integration by parts. We have
\begin{eqnarray*}
\lefteqn{ D^{\alpha}_x D^{\gamma}_y \Big( K_n(x,y) - K_{n+1}(x,y) \Big) = \sum_{\alpha_1+\alpha_2+\alpha_3=\alpha} \sum_{\mu \leq \gamma} \frac{\alpha!}{\alpha_1! \alpha_2! \alpha_3!} \binom{\gamma}{\mu} (-1)^{|\mu|} \times} \\
&& \times \int_{\mathbb{R}^d} e^{i(x-y)\xi} \xi^{\alpha_1+\mu} D^{\alpha_2}_x D^{\gamma-\mu}_y a(x,y,\xi) D^{\alpha_3}_x\Big( \Psi\big( \frac{(x,\xi)}{2^n}\big) - \Psi\big(\frac{(x,\xi)}{2^{n+1}}\big)\Big) d\xi.
\end{eqnarray*} We fix $\lambda \in \mathbb{N}$ and take $k>\lambda$ to determine later. We integrate by parts $N$ times, $N \in \mathbb{N}$, to get
\begin{equation}\label{int-KN}
\begin{split}
\lefteqn{ D^{\alpha}_x D^{\gamma}_y \Big( K_n(x,y) - K_{n+1}(x,y) \Big) = \sum_{\alpha_1+\alpha_2+\alpha_3=\alpha} \sum_{\mu \leq \gamma} \frac{\alpha!}{\alpha_1! \alpha_2! \alpha_3!} \binom{\gamma}{\mu}
\frac{(-1)^{N+|\mu|}}{|x_l-y_l|^N} \times }\\
& \ \ \ \ \ \ \times \int_{\mathbb{R}^d} e^{i(x-y)\xi} D^{N}_{\xi_l} \Big[ \xi^{\alpha_1+\mu} D^{\alpha_2}_x D^{\gamma-\mu}_y a(x,y,\xi) D^{\alpha_3}_x\Big( \Psi\big( \frac{(x,\xi)}{2^n}\big) - \Psi\big(\frac{(x,\xi)}{2^{n+1}}\big)\Big)
\Big] d\xi \\
& \ \ \ = \sum_{\substack{\alpha_1+\alpha_2+\alpha_3=\alpha \\ \mu \leq \gamma}} \sum_{\substack{N_1+N_2+N_3=N \\ N_1 \leq (\alpha_1)_l + \mu_l}} \frac{\alpha!}{\alpha_1! \alpha_2! \alpha_3!} \binom{\gamma}{\mu}
\frac{(-1)^{N+|\mu|}}{|x_l-y_l|^N} \frac{N!}{N_1! N_2! N_3!} \frac{\big((\alpha_1)_l + \mu_l\big)!}{\big((\alpha_1)_l + \mu_l - N_1 \big)!} \times \\
& \ \ \ \ \ \  \times \int_{\mathbb{R}^d} e^{i(x-y)\xi} \xi^{\alpha_1+\mu-N_1e_l} D^{\alpha_2}_x D^{\gamma-\mu}_y D^{N_2}_{\xi_l} a(x,y,\xi) D^{\alpha_3}_x D^{N_3}_{\xi_l} \big( \textstyle\Psi\big(\frac{(x,\xi)}{2^n}\big) -
\Psi\big(\frac{(x,\xi)}{2^{n+1}}\big)\big) d\xi.\nonumber
\end{split}
\end{equation}
Now, we integrate by parts again using an ultradifferential operator $G(D)$ as in the proof of Proposition~\ref{LemmaIteratedIntegralCauchy}. For a suitable power $G(D)^s$ of $G(D)$, $s \in \mathbb{N}$ depending on $\lambda$ to be determined, we use the formula
\[ e^{i(x-y)\xi} = \frac1{G^{s}(y-x)} G^{s}\big(-D_{\xi}\big) e^{i(x-y)\xi}\] to obtain 
\begin{equation}\label{multisum}
\begin{split}
\lefteqn{D^{\alpha}_x D^{\gamma}_y \Big( K_n(x,y) - K_{n+1}(x,y) \Big) = \sum_{\substack{\alpha_1+\alpha_2+\alpha_3=\alpha \\ \mu \leq \gamma}} \frac{\alpha!}{\alpha_1! \alpha_2! \alpha_3!} \binom{\gamma}{\mu}
\frac{(-1)^{N+|\mu|}}{|x_l-y_l|^N} \times }\\
& \ \ \ \times \sum_{\delta \in\N_{0}^d} b_{\delta} \sum_{\substack{N_1+N_2+N_3=N \\ N_1 \leq (\alpha_1)_l + \mu_l}} \frac{N!}{N_1! N_2! N_3!} \sum_{\substack{\delta_1+\delta_2+\delta_3=\delta \\ \delta_1 \leq \alpha_1+\mu-N_1e_l}}
\frac{\delta!}{\delta_1! \delta_2! \delta_3!} \frac{(\alpha_1+\mu-N_1e_l)!}{(\alpha_1+\mu-N_1e_l-\delta_1)!} \times \\ 
& \ \ \ \times \frac{\big((\alpha_1)_l + \mu_l\big)!}{\big((\alpha_1)_l + \mu_l - N_1 \big)!} \frac1{G^s(y-x)} \int e^{i(x-y)\xi} \xi^{\alpha_1+\mu-N_1e_l-\delta_1} D^{\alpha_2}_x D^{\gamma-\mu}_y D^{N_2e_l+\delta_2}_{\xi} a(x,y,\xi)
\times \\ 
& \ \ \ \times D^{\alpha_3}_x D^{N_3e_l+\delta_3}_{\xi} \Big( \Psi\big( \frac{(x,\xi)}{2^n}\big) - \Psi\big(\frac{(x,\xi)}{2^{n+1}}\big)\Big) d\xi.
\end{split}
\end{equation}
We know by the properties of $G^{s}(D)$ (formulas~\eqref{EqEstimationBAlpha} and~\eqref{EqEstimationDerivativesUltraOperator}) that there exist $D,C_1,C_2>0$ depending on $G$ such that
\begin{eqnarray*}
|b_{\delta}| &\leq& e^{sD} e^{-sD \varphi^{\ast}\big(\frac{|\delta|}{sD}\big)}, \\
\Big| \frac1{G^s(y-x)}\Big| &\leq& C_1^s e^{-sC_2 \omega(y-x)}.
\end{eqnarray*} Here we set $A^2=\frac1{c_0^2}+d$ and $\widetilde{p} \in \mathbb{N}$ so that $\max\{ \sqrt{2}A, 6 \} \leq e^{\widetilde{p}\rho}$. By the definition of amplitude, there exists a constant $C_k>0$ such that
\begin{eqnarray*}
&& |D^{\alpha_2}_x D^{\gamma-\mu}_y D^{N_2e_l + \delta_2}_{\xi} a(x,y,\xi)| \\
&& \ \ \  \leq C_k \Big(\frac{\langle x-y \rangle}{\langle(x,y,\xi)\rangle}\Big)^{\rho(|\alpha_2+\gamma-\mu+\delta_2|+N_2)} e^{4k\rho L^{2\widetilde{p}+3} \varphi^{\ast}\big(\frac{|\alpha_2+\gamma-\mu+\delta_2|+N_2}{4kL^{2\widetilde{p}+3}}\big)} e^{m\omega(x)}
e^{m\omega(y)} e^{m\omega(\xi)}.
\end{eqnarray*}
Now, we observe that 
the support of
$\Psi\big(\frac{x}{2^n},\frac{\xi}{2^n}\big) - \Psi\big(\frac{x}{2^{n+1}},\frac{\xi}{2^{n+1}}\big)$ is  in the set
$B_{n}:=\{ (x,\xi) \in \mathbb{R}^{2d} : 2^n \leq \langle(x,\xi)\rangle \leq 3\cdot2^{n+1} \}.$
Hence, we have, for $k\in\N$ depending on $\lambda, s$ to be chosen later, and for the selection of $\widetilde{p}$ (Lemma~\ref{lemmafistrella} (1)),
\begin{eqnarray*}
\lefteqn{\Big| D^{\alpha_3}_x D^{N_3e_l + \delta_3}_{\xi} \Big( \Psi\big(\frac{(x,\xi)}{2^n}\big) - \Psi\big(\frac{(x,\xi)}{2^{n+1}}\big)\Big) \Big|}\\ &&\le  2D_k e^{4k\rho L^{2\widetilde{p}+4}
\varphi^{\ast}\big(\frac{|\alpha_3+\delta_3|+N_3}{4kL^{2\widetilde{p}+4}}\big)} \Big(\frac1{2^n}\Big)^{|\alpha_3+\delta_3|+N_3}  \\
&&\le 2D_k e^{4k\rho L^{2\widetilde{p}+4}} e^{4k\rho L^{2\widetilde{p}+3} \varphi^{\ast}\big(\frac{|\alpha_3+\delta_3|+N_3}{4kL^{2\widetilde{p}+3}}\big)} \frac1{\langle(x,\xi)\rangle^{\rho N_3}}.
\end{eqnarray*} 
On the other hand, we also have according to~\eqref{EqLemmaTechnical1} (observe that $|\alpha_1+\mu|-N_1-|\delta_1|\ge 0$ by \eqref{multisum}), 
\begin{eqnarray*}
|\xi^{\alpha_1+\mu-N_1e_l-\delta_1}| \leq |\xi|^{|\alpha_1+\mu|-N_1-|\delta_1|} \leq \langle(x,\xi)\rangle^{|\alpha_1+\mu|-N_1-|\delta_1|} 
\leq \frac{e^{\lambda L^4\varphi^{\ast}\big(\frac{|\alpha_1+\mu|}{\lambda L^4}\big)} e^{\lambda L^4\omega(\langle(x,\xi)\rangle)}}{\langle(x,\xi)\rangle^{\rho N_1}}.
\end{eqnarray*} 
Moreover, since $|x_l-y_l|\geq c_0$, we get
\[ \langle x-y \rangle^2 \leq 1+d|x_l-y_l|^2 \leq \frac{|x_l-y_l|^2}{c_0^2}+d|x_l-y_l|^2 = A^2|x_l-y_l|^2, \] with $A$ defined previously. Thus $\langle x-y \rangle \leq A|x_l-y_l|$. Therefore, by Lemma~\ref{LemmaCorcheteIneq}, we
obtain (remember that $\mu\le \gamma$ from \eqref{multisum})
\begin{eqnarray*}
\lefteqn{\Big( \frac{\langle x-y \rangle}{\langle(x,y,\xi)\rangle}\Big)^{\rho(|\alpha_2+\gamma-\mu+\delta_2|+N_2)} \frac1{|x_l-y_l|^{N}} \leq \sqrt{2}^{|\alpha_2+\gamma-\mu+\delta_2|} \Big( \frac{\langle x-y
\rangle}{\langle(x,y,\xi)\rangle} \Big)^{\rho N_2} \frac1{|x_l-y_l|^N} }\\
&& \leq \sqrt{2}^{|\alpha_2+\gamma-\mu+\delta_2|} \frac1{\langle(x,\xi)\rangle^{\rho N_2}} \frac{\langle x-y \rangle^{N}}{|x_l-y_l|^{N}} 
 \leq (\sqrt{2}A)^{|\alpha_2+\gamma-\mu+\delta_2|+N} \frac1{\langle(x,\xi)\rangle^{\rho N_2}}.
\end{eqnarray*} We also have, by Proposition~\ref{PropTechnical} and Lemma~\ref{Lemma158TesisDavid}, 
\begin{eqnarray*}
\frac{(\alpha_1+\mu-N_1e_l)!}{(\alpha_1+\mu-N_1e_l-\delta_1)!} \frac{\big((\alpha_1)_l + \mu_l\big)!}{\big((\alpha_1)_l + \mu_l - N_1 \big)!} &\leq& 2^{|\alpha_1+\mu|-N_1} 2^{(\alpha_1)_l+\mu_l} \delta_1! N_1! \\
&\leq& 4^{|\alpha_1+\mu|} E_k e^{kL^3 \varphi^{\ast}\big(\frac{|\delta_1|}{kL^3}\big)} e^{2kL^{2\widetilde{p}}\varphi^{\ast}_{\sigma}\big(\frac{N_1}{2kL^{2\widetilde{p}}}\big)} \\
&\leq& 4^{|\alpha_1+\mu|} E'_k e^{kL^3 \varphi^{\ast}\big(\frac{|\delta_1|}{kL^3}\big)} e^{2k\rho L^{2\widetilde{p}}\varphi^{\ast}\big(\frac{N_1}{2kL^{2\widetilde{p}}}\big)}.
\end{eqnarray*}
Proceeding as in previous proofs we obtain, for some constant $C'_{\lambda,k,s}>0$ depending on $\lambda$, $k$ and $s$, and $(x,\xi)\in B_{n}$,
\begin{equation}\label{KN}
\begin{split}
\lefteqn{\Big| D^{\alpha}_x D^{\gamma}_y \Big( K_n(x,y) - K_{n+1}(x,y) \Big) \Big| }\\
&\ \ \ \le {C'_{\lambda,k,s} 3^{|\alpha|} 2^{|\gamma|} e^{\lambda L^2 \varphi^{\ast}\big(\frac{|\alpha+\gamma|}{\lambda L^2}\big)} \Big( \sum_{\delta \in\N_{0}^{d}} e^{-sD \varphi^{\ast}\big(\frac{|\delta|}{sD}\big)} e^{kL^3
\varphi^{\ast}\big( \frac{|\delta|}{kL^3}\big)} 3^{|\delta|} \Big) e^{-sC_2\omega(y-x)} \times }\\
&\ \ \  \qquad\times \int e^{\lambda L^4 \omega(\langle(x,\xi)\rangle)} \langle(x,\xi)\rangle^{-\rho N} 3^N e^{2k\rho L^{\widetilde{p}} \varphi^{\ast}\big(\frac{N}{2kL^{\widetilde{p}}}\big)} e^{m\omega(x)} e^{m\omega(y)} e^{m\omega(\xi)} d\xi.
\end{split}
\end{equation}
%
Since the inequality \eqref{KN} holds for every $N \in \mathbb{N}$, we can take the infimum in $N$ to obtain, by formula~\eqref{EqLemmaTechnical2}, for some constant $C>0,$
\begin{eqnarray*}
\inf_{N \in \mathbb{N}_0} \langle(x,\xi)\rangle^{-\rho N} e^{2k\rho \varphi^{\ast}\big(\frac{N}{2k}\big)} &=& \Big( \inf_{N \in \mathbb{N}_0} \langle(x,\xi)\rangle^{-N} e^{2k\varphi^{\ast}\big(\frac{N}{2k}\big)} \Big)^{\rho} \\ &\leq& e^{-2k\rho \omega(\langle(x,\xi)\rangle) + \rho\log(\langle(x,\xi)\rangle)} \\
&\leq& C e^{-\rho(2k-1)\omega(\langle(x,\xi)\rangle)} \\
&\leq& C e^{-2\rho(k-1)\omega(\langle(x,\xi)\rangle)} e^{-\rho \omega(2^n)}.
\end{eqnarray*}
%
%
If we take $s>0$ big enough and $k \geq sD$, the series in \eqref{KN} is convergent (proceeding  as in \eqref{series1} and \eqref{series2}) and, hence
%
 we can deduce that for each $\lambda>0$ there is some constant $C_{\lambda}>0$ such that
\begin{eqnarray*}
\big| D^{\alpha}_x D^{\gamma}_y \big( K_n(x,y) - K_{n+1}(x,y) \big)\big| &\leq& C_{\lambda} e^{\lambda \varphi^{\ast}\big(\frac{|\alpha+\gamma|}{\lambda}\big)} e^{-\lambda \omega(x)} e^{-\lambda \omega(y)} e^{-\rho\omega(2^n)},
\end{eqnarray*} for every $(x,y)\notin \Delta_{r}$.

Let $\chi$ be as in Lemma~\ref{LemmaExistenceFunctionEOmega}. It is clear that $\{\chi K_n\}$ is a Cauchy sequence in $\mathcal{S}_{\omega}\big(\mathbb{R}^{2d}\big)$. Since $\mathcal{S}_{\omega}\big(\mathbb{R}^{2d}\big)$ is complete, there exists $T \in
\mathcal{S}_{\omega}\big(\mathbb{R}^{2d}\big)$ such that $\chi K_n \to T$ in $\mathcal{S}_{\omega}\big(\mathbb{R}^{2d}\big)$. On the other hand, we have seen that $K_n \to K$ in $\mathcal{S}'_{\omega}\big(\mathbb{R}^{2d}\big)$. Hence, $\chi
K_n \to \chi K$ in $\mathcal{S}'_{\omega}\big(\mathbb{R}^{2d}\big)$ when $n\to\infty$. This shows that $\chi K = T$ in $\mathcal{S}'_{\omega}\big(\mathbb{R}^{2d}\big)$. Since $K=T$ in $\mathbb{R}^{2d} \setminus \Delta_{r}$, we
have
\[ |D^{\alpha}_x D^{\beta}_y K(x,y)| = |D^{\alpha}_x D^{\beta}_y T(x,y)| \leq C_{\lambda} e^{\lambda \varphi^{\ast}\big(\frac{|\alpha+\beta|}{\lambda}\big)} e^{-\lambda \omega(x)} e^{-\lambda \omega(y)}, \] for $(x,y) \in
\mathbb{R}^{2d} \setminus \Delta_{r}$, which completes the proof.
\end{proof}
We observe that the constant $C_{\lambda}$ at the end of the proof of the last result becomes larger when $r>0$ becomes smaller.

\subsection{Composition of pseudodifferential operators and the transpose operator}

Now, for simplicity, in what follows we denote $\mathcal{S}_{\omega}$ for $\mathcal{S}_{\omega}(\R^{d}).$ The following lemma is taken from~\cite[Lemma 3.11]{FGJ2005pseudo}.

\begin{lema}\label{LemmaPrevioLarguisimo}
Let $m,n,j\in\N$ and $t>0$ such that $m \geq n$ and $\frac1{e} e^{\frac{m}{j} \varphi^{\ast}(\frac{j}{m})} \leq t \leq e^{\frac{n}{j} \varphi^{\ast}(\frac{j}{n})}$. We have
\[ t^{j+1} \geq e^{n\omega(t)} e^{2m \varphi^{\ast}(\frac{j}{2m})} e^{-j}. \] In particular,
\[ e^{n \varphi^{\ast}(\frac{j}{n})} \geq e^{(n-1)\omega(t)} e^{2n \varphi^{\ast}(\frac{j}{2n})}, \] for $j$ large enough.
\end{lema}

\begin{theo}\label{TheoLarguisimo}
Let $a(x,y,\xi)$ be an amplitude in $\GA^{m,\omega}_{\rho}$ with associated pseudodifferential operator  $A$. Then there exist a pseudodifferential operator $P(x,D)$ given by a symbol $p(x,\xi)$ in $\GS^{2m,\omega}_{\rho}$ and an $\omega$-regularizing operator $\widetilde R$ such that $Au = P(x,D)u + \widetilde{R}u$, for each $u \in \mathcal{S}_{\omega}$ and, moreover,
\[ p(x,\xi) \sim \sum_{j=0}^{\infty} p_j(x,\xi) = \sum_{j=0}^{\infty} \sum_{|\alpha|=j} \frac1{\alpha!} D^{\alpha}_{\xi} \partial^{\alpha}_y a(x,y,\xi)\left. \right|_{y=x}.\]
\end{theo}
\begin{proof}
First of all we consider $\chi(x,y)$ from Lemma~\ref{LemmaExistenceFunctionEOmega}. We then decompose $a(x,y,\xi)$ as
\[ a(x,y,\xi) = \chi(x,y) a(x,y,\xi) + \big(1-\chi(x,y)\big)a(x,y,\xi). \] On the one hand, it follows from Theorem~\ref{TheoFormalKernel} and Proposition~\ref{PropExampleExtensionPseudodifferentialOperator} that $\chi(x,y)a(x,y,\xi) \in \GA^{m,\omega}_{\rho}$ defines an $\omega$-regularizing operator. Then we can suppose that the support of the amplitude is in $\Delta_k \times \mathbb{R}^d$ for some $k>0$. 

We have $\sum_j p_j \in \FGS^{2m,\omega}_{\rho}$, by Example~\ref{ExamFormalSum1}. Let $(j_n)_{n}$ be as in the proof of Theorem~\ref{TheoEquivalentSymbol} with $\frac{n}{j}\varphi^{\ast}\big(\frac{j}{n}\big) \geq \max\big\{n, \log(C_{2n}), \log(D_n) \big\}$, where $C_n$ and $D_n$ are the constants, depending on $n\in\N$, which appear in Definition~\ref{amplitude} (of amplitude) and the definition of formal sum for $\sum_{j=0}^{\infty} p_j(x,\xi)$.  We take
\[ p(x,\xi) = \sum_{j=0}^{\infty} \varphi_j(x,\xi) p_j(x,\xi), \] where $\varphi_{j}$ is defined in \eqref{not18}. We denote $P:=P(x,D)$. By Theorem~\ref{TheoEquivalentSymbol},  $p \sim \sum p_j$. By Lemma~\ref{LemmaSumOfAmplitudes}, we have $A=\sum_{N=0}^{\infty}A_N$, where $A_N$ is the pseudodifferential operator with amplitude $a(x,y,\xi)\big( \varphi_N - \varphi_{N+1} \big)(x,\xi)$. Moreover, by Proposition~\ref{LemmaPseudoDifferentialOperatorLimit}, $P= \lim_{N \to \infty} S_N$ in $L(\mathcal{S}_{\omega},\mathcal{S}'_{\omega})$, where $S_N$ is the pseudodifferential operator with symbol $\sum_{j=0}^N \big(\varphi_j - \varphi_{j+1}\big)(x,\xi)\big( \sum_{l=0}^j p_l(x,\xi) \big)$.

That is, for $u \in \mathcal{S}_{\omega}$, we have
\[ Au(x) = \sum_{N=0}^{\infty} A_Nu(x) = \sum_{N=0}^{\infty} \iint e^{i(x-y)\xi} \Big( (\varphi_N - \varphi_{N+1})(x,\xi) a(x,y,\xi) \Big) u(y) dy d\xi, \] and
\[ Pu(x) = \lim_{N \to \infty} \iint e^{i(x-y)\xi} \Big(\sum_{j=0}^N \big( \varphi_j - \varphi_{j+1} \big)(x,\xi) \big(\sum_{l=0}^j p_l(x,\xi) \big)\Big) u(y) dy d\xi. \] Thus, we can write $A-P$ as the series $\sum_{N=0}^{\infty} P_N$, where $P_N$ is the pseudodifferential operator associated to
\[ \widetilde{a}_N(x,y,\xi) = \big( \varphi_{N} - \varphi_{N+1}\big)(x,\xi) \Big( a(x,y,\xi) - \sum_{j=0}^N p_j(x,\xi) \Big), \] which is an amplitude. Our purpose is to show that the formal kernel \[K(x,y):=\sum_{N=0}^{\infty} \int e^{i(x-y)\xi} \widetilde{a}_N(x,y,\xi) d\xi \] belongs to $\mathcal{S}_{\omega}(\R^{2d})$. We denote $K_j(x,y)=\int e^{i(x-y)\xi} \widetilde{a}_j(x,y,\xi) d\xi$.

As in~\cite[Theorem 3.13]{FGJ2005pseudo}, we can write the kernel $K$ as the limit when $N \to \infty$ of
\[ \sum_{j=1}^N K_j = \sum_{j=1}^N I_j + \sum_{j=1}^N Q_j - W_N, \] where
\begin{eqnarray*}
&& I_j(x,y) = \sum_{|\alpha|=j} \sum_{0 \neq \beta \leq \alpha} \frac1{\beta! (\alpha-\beta)!} \int e^{i(x-y)\xi} D^{\beta}_{\xi} \varphi_j(x,\xi) D^{\alpha-\beta}_{\xi} \partial^{\alpha}_y a(x,x,\xi) d\xi \\
&& Q_j(x,y) = \sum_{|\alpha|=j+1} \sum_{\beta \leq \alpha} \frac1{\beta! (\alpha-\beta)!} \int e^{i(x-y)\xi} D^{\beta}_{\xi}\big(\varphi_j(x,\xi) - \varphi_{j+1}(x,\xi)\big) D^{\alpha-\beta}_{\xi} \omega_{\alpha}(x,y,\xi) d\xi \\
&& \omega_{\alpha}(x,y,\xi) = (j+1) \int_0^1 \partial^{\alpha}_y a(x,x+t(y-x),\xi)(1-t)^j dt \\
&& W_N(x,y) = \sum_{|\alpha|=1}^N \sum_{0 \neq \beta \leq \alpha} \frac1{\beta! (\alpha-\beta)!} \int e^{i(x-y)\xi} D^{\beta}_{\xi} \varphi_{N+1}(x,\xi) D^{\alpha-\beta}_{\xi} \partial^{\alpha}_y a(x,x,\xi) d\xi.
\end{eqnarray*}

We will not give a detailed proof of all the steps below, unless it was necessary.  
\vskip.5\baselineskip
\emph{\underline{First step}}. We see that $\sum_{j=1}^{\infty} I_j$ belongs to $\mathcal{S}_{\omega}(\R^{2d})$. To this, we consider $\gamma,\epsilon \in \mathbb{N}_0^d$. We begin by differentiating  $I_j$:
\begin{eqnarray*}
\lefteqn{D^{\gamma}_x D^{\epsilon}_y I_j(x,y) = \sum_{|\alpha|=j} \sum_{0 \neq \beta \leq \alpha} \frac1{\beta! (\alpha-\beta)!} \sum_{\gamma_1+\gamma_2+\gamma_3=\gamma} \frac{\gamma!}{\gamma_1! \gamma_2! \gamma_3!} \times }\\
&& \times \int (-1)^{\epsilon} \xi^{\epsilon + \gamma_1} e^{i(x-y)\xi} D^{\gamma_2}_x D^{\beta}_{\xi} \varphi_j(x,\xi) D^{\gamma_3}_x \partial^{\alpha}_y D^{\alpha-\beta}_{\xi} a(x,x,\xi) d\xi.
\end{eqnarray*}
Here we use integration by parts with the formula
\begin{equation}\label{EqIntegrationByPartsDxi}
e^{i(x-y)\xi} = \frac1{G(y-x)}G\big(-D_{\xi}\big) e^{i(x-y)\xi},
\end{equation}
for a suitable power $G^{s}(D)$ of $G(D)$, being $G(\xi)$ the function that appears in Theorem~\ref{TheoLangenbruchSeveralVariables}, to obtain
\begin{eqnarray*}
\lefteqn{\int \xi^{\epsilon + \gamma_1} e^{i(x-y)\xi} D^{\gamma_2}_x D^{\beta}_{\xi} \varphi_j(x,\xi) D^{\gamma_3}_x \partial^{\alpha}_y D^{\alpha-\beta}_{\xi} a(x,x,\xi) d\xi}\\
&&=\int e^{i(x-y)\xi} \frac1{G^s(y-x)} G^{s}(D_{\xi}) \big\{ \xi^{\epsilon+\gamma_1} D^{\gamma_2}_x D^{\beta}_{\xi} \varphi_j(x,\xi) D^{\gamma_3}_x \partial^{\alpha}_y D^{\alpha-\beta}_{\xi} a(x,x,\xi) \big\} d\xi \\
&& = \int e^{i(x-y)\xi} \frac1{G^s(y-x)} \sum_{\tau \in\N_{0}^{d}} b_{\tau} \sum_{\substack{\tau_1+\tau_2+\tau_3=\tau\\ \tau_{1}\le \epsilon+\gamma_{1}}} \frac{\tau!}{\tau_1! \tau_2! \tau_3!} \frac{(\epsilon+\gamma_1)!}{(\epsilon+\gamma_1-\tau_1)!} \xi^{\epsilon+\gamma_1-\tau_1} \times \\
&&\qquad \times D^{\gamma_2}_x D^{\beta+\tau_2}_{\xi} \varphi_j(x,\xi) D^{\gamma_3}_x \partial^{\alpha}_y D^{\alpha-\beta+\tau_3}_{\xi} a(x,x,\xi)d\xi.
\end{eqnarray*}
Therefore
\begin{eqnarray*}
\lefteqn{D^{\gamma}_x D^{\epsilon}_y I_j(x,y)}\\
&& =\sum_{|\alpha|=j} \sum_{0 \neq \beta \leq \alpha} \frac1{\beta! (\alpha-\beta)!} (-1)^{\epsilon} \sum_{\tau 
\in\N_{0}^{d}} b_{\tau} \sum_{\substack{\gamma_1+\gamma_2+\gamma_3=\gamma \\ \tau_1 + \tau_2 + \tau_3 = \tau\\ \tau_{1}\le \epsilon+\gamma_{1}}} \frac{\gamma!}{\gamma_1! \gamma_2! \gamma_3!} \frac{\tau!}{\tau_1! \tau_2! \tau_3!} \frac{(\epsilon+\gamma_1)!}{(\epsilon+\gamma_1-\tau_1)!} \times \\
&& \qquad\times \frac1{G^s(y-x)} \int e^{i(x-y)\xi} \xi^{\epsilon+\gamma_1-\tau_1} D^{\gamma_2}_x D^{\beta+\tau_2}_{\xi} \varphi_j(x,\xi) D^{\gamma_3}_x \partial^{\alpha}_y D^{\alpha-\beta+\tau_3}_{\xi} a(x,x,\xi) d\xi.
\end{eqnarray*}
Now, proceeding as in \cite[Theorem 3.13]{FGJ2005pseudo} (using Lemma~\ref{LemmaPrevioLarguisimo}), it follows that $\sum_{j=1}^{\infty} I_j \in \mathcal{S}_{\omega}(\R^{2d})$.

\vskip.5\baselineskip

\emph{\underline{Second step.}} Now, let us prove that $\sum_{j=1}^{\infty} Q_j$ belongs to $\mathcal{S}_{\omega}(\R^{2d})$. We proceed as before, and we first calculate, for $\gamma,\epsilon \in \mathbb{N}_0^d$, the derivatives of $Q_j$:
\begin{eqnarray*}
&& D^{\gamma}_x D^{\epsilon}_y Q_j(x,y) = \sum_{|\alpha|=j+1} \sum_{\beta \leq \alpha} \frac1{\beta! (\alpha-\beta)!} \sum_{\substack{\gamma_1 + \gamma_2 + \gamma_3 = \gamma \\ \epsilon_1 + \epsilon_2 = \epsilon}} \frac{\gamma!}{\gamma_1! \gamma_2! \gamma_3!} \frac{\epsilon!}{\epsilon_1! \epsilon_2!} (-1)^{\epsilon_1} \times \\
&& \ \ \  \times \int \xi^{\gamma_1 + \epsilon_1} e^{i(x-y)\xi} D^{\gamma_3}_x D^{\beta}_{\xi} \big( \varphi_j(x,\xi) - \varphi_{j+1}(x,\xi) \big) D^{\gamma_2}_x D^{\epsilon_2}_y D^{\alpha-\beta}_{\xi} \omega_{\alpha}(x,y,\xi) d\xi.
\end{eqnarray*} We use again the integration by parts given by   formula~\eqref{EqIntegrationByPartsDxi} with $G^{s}(D)$ for a suitable power of $G(D)$ in the integral above to obtain, in the integrand,
\begin{eqnarray*}
&& e^{i(x-y)\xi} \frac1{G^s(y-x)} G^s(D_{\xi}) \Big( \xi^{\gamma_1+\epsilon_1} D^{\gamma_3}_x D^{\beta}_{\xi}\big( \varphi_j(x,\xi) - \varphi_{j+1}(x,\xi) \big) D^{\gamma_2}_x D^{\epsilon_2}_y D^{\alpha-\beta}_{\xi} \omega_{\alpha}(x,y,\xi) \Big)  \\
&& \ \ \  = e^{i(x-y)\xi} \frac1{G^s(y-x)} \sum_{\tau\in\N_{0}^{d}} b_{\tau} \sum_{\substack{\tau_1+\tau_2+\tau_3=\tau\\\tau_{1}\le \gamma_{1}+\epsilon_{1}}} \frac{\tau!}{\tau_1! \tau_2! \tau_3!} \frac{(\gamma_1+\epsilon_1)!}{(\gamma_1+\epsilon_1-\tau_1)!} \xi^{\gamma_1+\epsilon_1-\tau_1} \times \\
&& \ \ \ \ \ \times D^{\gamma_3}_x D^{\beta+\tau_3}_{\xi}\big(\varphi_j(x,\xi) - \varphi_{j+1}(x,\xi)\big) D^{\gamma_2}_x D^{\epsilon_2}_y D^{\alpha-\beta+\tau_2}_{\xi} \omega_{\alpha}(x,y,\xi).
\end{eqnarray*} 
So, we have
\begin{equation*}
\begin{split}
D^{\gamma}_x D^{\epsilon}_y Q_j(x,y)&=\sum_{|\alpha|=j+1} \sum_{\beta \leq \alpha} \frac1{\beta! (\alpha-\beta)!} \sum_{\tau \in\N_{0}^{d}} b_{\tau} \sum_{\substack{\gamma_1+\gamma_2+\gamma_3 = \gamma \\ \epsilon_1+\epsilon_2=\epsilon \\ \tau_1+\tau_2+\tau_3=\tau,\,\tau_{1}\le \gamma_{1}+\epsilon_{1} }} (-1)^{\epsilon_1} \frac{\gamma!}{\gamma_1! \gamma_2! \gamma_3!} \frac{\epsilon!}{\epsilon_1! \epsilon_2!} \frac{\tau!}{\tau_1! \tau_2! \tau_3!}\times\\
&\ \ \ \ \ \ \ \ \times \frac{(\gamma_1+\epsilon_1)!}{(\gamma_1+\epsilon_1-\tau_1)!} \frac1{G^s(y-x)} \int e^{i(x-y)\xi} \xi^{\gamma_1+\epsilon_1-\tau_1}\times  \\
&\ \ \ \ \ \ \ \ \times D^{\gamma_3}_x D^{\beta+\tau_3}_{\xi} \big( \varphi_j(x,\xi) - \varphi_{j+1}(x,\xi) \big) D^{\gamma_2}_x D^{\epsilon_2}_y D^{\alpha-\beta+\tau_2}_{\xi} \omega_{\alpha}(x,y,\xi) d\xi.
\end{split}
\end{equation*}
Now, fix $\lambda>0$ and take $n\geq \frac{\lambda}{\rho}$. 
To estimate the derivatives of $\omega_{\alpha}(x,y,\xi)$, we denote by $p$ a positive number such that $1+k<e^p$, where $k$ is the constant for the subscript of  $\Delta_k$. For this $n$, we define $\overline{n}:=16(2nL^{3+q}+2\frac{m}{\rho}L^q+1)L^p$ where $L>0$ is the constant in Lemma~\ref{lemmafistrella}\,(1), $R>0$ is the constant in Definition~\ref{DefFormalSums} and $q\in\N$ is such that $2^{q}\ge 3R$. Then, there is $C_n>0$ such that 
\begin{eqnarray*}
\lefteqn{|D^{\gamma_2}_x D^{\epsilon_2}_y D^{\alpha-\beta+\tau_2}_{\xi} \omega_{\alpha}(x,y,\xi)|} \\
&& \leq (j+1)\int_0^1 |1-t|^j |t|^{|\epsilon_2|} \big| D^{\gamma_2}_x D^{\alpha+\epsilon_2}_y D^{\alpha-\beta+\tau_2}_{\xi} a(x,x+t(y-x),\xi) \big| dt \\
&& \leq C_n e^{\overline{n} \rho \varphi^{\ast}\big(\frac{|2\alpha-\beta+\gamma_2+\epsilon_2+\tau_2|}{\overline{n}}\big)} e^{m\omega(x)} e^{m\omega(\xi)} (j+1)  \times \\
&&\ \ \ \ \times \int_0^1 |1-t|^j |t|^{|\epsilon_2|} \Big( \frac{\langle t(x-y) \rangle}{\langle(x,x+t(y-x),\xi)\rangle} \Big)^{\rho|2\alpha-\beta+\gamma_2+\epsilon_2+\tau_2|} e^{m\omega(x+t(y-x))} dt.
\end{eqnarray*}

Since $|x-y| < k$, we have $\langle t(x-y) \rangle \leq \sqrt{1+k^2}<1+k<e^{p}$. Also, $\langle(x,x+t(y-x),\xi)\rangle \geq \langle(x,\xi)\rangle$. 
With this, we argue as in the 
first step to see that $\sum_{j=1}^{\infty} Q_j$ belongs to $\mathcal{S}_{\omega}(\R^{2d})$ for $R\ge 1$ big enough.

\vskip\baselineskip

\underline{\em Third step.}  Let $T_N:\mathcal{S}_{\omega} \to \mathcal{S}_{\omega}$ be the operator with kernel $W_N$. Since $A-P = \sum_{N=0}^{\infty} P_N$ converges in $L(\mathcal{S}_{\omega},\mathcal{S}'_{\omega})$, it follows that $(T_N)$ converges to an operator $T:\mathcal{S}_{\omega} \to \mathcal{S}_{\omega}$ in $L(\mathcal{S}_{\omega},\mathcal{S}'_{\omega})$. In fact, we have seen that $\sum_{j=1}^N I_j + \sum_{j=1}^N Q_j$ converges in $\mathcal{S}_{\omega}(\R^{2d})$ as $N \to +\infty$, hence in $\mathcal{S}'_{\omega}(\R^{2d})$. Then, by the kernel's theorem, $\sum_{j=1}^N I_j + \sum_{j=1}^N Q_j$ is the kernel of an operator that converges in $L(\mathcal{S}_{\omega},\mathcal{S}'_{\omega})$ as $N \to \infty$.

We want to show $T=0$ in $L(\mathcal{S}_{\omega},\mathcal{S}'_{\omega})$. To this aim, we fix $N \in \mathbb{N}$, $j_n \leq N+1 < j_{n+1}$ and we set $a_N := R e^{\frac{n}{N+1} \varphi^{\ast}\big(\frac{N+1}{n}\big)}$. We assume $2a_N \leq \langle(x,\xi)\rangle \leq 3a_N$ since otherwise $D^{\beta}_{\xi} \varphi_{N+1}(x,\xi)$ vanishes for all $\beta\neq 0$. For $f \in \mathcal{S}_{\omega}$, we have
\begin{eqnarray*}
\lefteqn{\langle T_Nu, f \rangle= \int T_Nu(x) f(x) dx = \int \Big( \int W_N(x,y) u(y) dy \Big) f(x) dx}\\
&&= \int \Big( \int \sum_{|\alpha|=1}^N \sum_{0 \neq \beta \leq \alpha} \frac1{\beta! (\alpha-\beta)!} \Big( \int e^{i(x-y)\xi} D^{\beta}_{\xi} \varphi_{N+1}(x,\xi) D^{\alpha-\beta}_{\xi} \partial^{\alpha}_y a(x,x,\xi) d\xi \Big) u(y) dy \Big) f(x) dx \\
&&= \sum_{|\alpha|=1}^N \sum_{0 \neq \beta \leq \alpha} \frac1{\beta! (\alpha-\beta)!} \int \Big( \int e^{ix\xi} D^{\beta}_{\xi} \varphi_{N+1}(x,\xi) D^{\alpha-\beta}_{\xi} \partial^{\alpha}_y a(x,x,\xi) \widehat{u}(\xi) d\xi \Big) f(x) dx.
\end{eqnarray*} 
By definition of amplitude and $\varphi_j$, for all $n\in\N$ there are $C_{n},D_{n}>0$ with 
\begin{eqnarray*}
\big| D^{\beta}_{\xi} \varphi_{N+1}(x,\xi) \big| &\leq& D_n \langle(x,\xi)\rangle^{-\rho|\beta|} e^{n\rho \varphi^{\ast}\big(\frac{|\beta|}{n}\big)} \\
\big| D^{\alpha-\beta}_{\xi} D^{\alpha}_y a(x,x,\xi) \big| &\leq& C_{2n} \langle(x,\xi)\rangle^{-\rho|2\alpha-\beta|} e^{2n\rho \varphi^{\ast}\big(\frac{|2\alpha-\beta|}{2n}\big)} e^{2m\omega(x)} e^{m\omega(\xi)}.
\end{eqnarray*} Since $u$ and $f$ belong to $\mathcal{S}_{\omega}$, by Definition~\ref{def3}, there exist $C_1,C_2>0$ (that only depend on $m$) such that 
\begin{eqnarray*}
\big| \widehat{u}(\xi) \big| \leq C_1 e^{-(m+1)\omega(\xi)} \quad\mbox{and}\quad
|f(x)| \leq C_2 e^{-(2m+1)\omega(x)}.
\end{eqnarray*} 
We observe that, by the convexity of $\varphi^{*}$, $e^{n\rho \varphi^{\ast}\big(\frac{|\beta|}{n}\big)} e^{2n \rho \varphi^{\ast}\big(\frac{|2\alpha-\beta|}{2n}\big)} \leq e^{2n \rho \varphi^{\ast}\big(\frac{|\alpha|}{n}\big)}.$ On the other hand, since $\varphi^{*}(x)/x$ is increasing,
\[ \langle(x,\xi)\rangle^{-\rho|\alpha|} \leq (2R)^{-\rho|\alpha|} e^{-\rho|\alpha|\frac{n}{N+1}\varphi^{\ast}\big(\frac{N+1}{n}\big)} \leq (2R)^{-\rho|\alpha|} e^{-\rho n \varphi^{\ast}\big(\frac{|\alpha|}{n}\big)}. \] These estimates give
\begin{eqnarray}\nonumber
\big| \langle T_Nu, f \rangle \big| &\le& \sum_{|\alpha|=1}^N \sum_{0 \neq \beta \leq \alpha} \frac1{\beta! (\alpha-\beta)!} \iint_{\langle(x,\xi)\rangle \geq 2a_N} C_{2n} D_n C_1 C_2 e^{2n\rho \varphi^{\ast}\big(\frac{|\alpha|}{n}\big)} (2R)^{-2\rho|\alpha|}  \times \\ \nonumber
&&\ \ \ \ \  \times e^{-2\rho n \varphi^{\ast}\big(\frac{|\alpha|}{n}\big)} e^{2m\omega(x)} e^{m\omega(\xi)} e^{-(2m+1)\omega(x)} e^{-(m+1)\omega(\xi)} d\xi dx \\ \label{Eq1}
&\le&  \sum_{|\alpha|=1}^N \sum_{0 \neq \beta \leq \alpha} \frac{(2R)^{-2\rho|\alpha|} E_{n}}{\beta! (\alpha-\beta)!}  \Big( \int_{\R^{d}\times \{|\xi| \geq 2a_N\}}+\int_{\{|x|\ge 2a_N\}\times \R^{d}}\Big) e^{-\omega(x)-\omega(\xi)} dx  d\xi.
\end{eqnarray} 
Here, $E_{n}=C_{1} C_{2} C_{2n} D_{n}$. Now, we consider in \eqref{Eq1} only the integral on $\R^{d}\times \{|\xi| \geq 2a_N\}$. The argument for the addend with the integral on $\{|x|\ge 2a_N\}\times \R^{d}$ is analogous. By property $(\gamma)$ of the weight function, $\int e^{-\omega(x)} dx$ converges and, moreover, for $N$ big enough, for some constant $C>0$, we also have
\[ \int_{|\xi| \geq 2a_N} e^{-\omega(\xi)} d\xi \leq \frac{C}{(2a_N)^3}. \] 
So, we obtain the estimate
\begin{eqnarray*}
CC_1 C_2 \Big( \int e^{-\omega(x)} dx \Big) \sum_{|\alpha|=1}^N \sum_{0 \neq \beta \leq \alpha} \frac1{\beta! (\alpha-\beta)!} \frac1{(2R)^{\rho|\alpha|}} \frac{C_{2n} D_n}{a_N^3}.
\end{eqnarray*} 
Finally, by the selection of $( j_n )$, we have $e^{-\frac{n}{j}\varphi^{\ast}\big(\frac{j}{n}\big)} \leq e^{-n}$ for $j \geq j_n$. This finishes the proof, since $\frac1{a_N} \frac{C_{2n}}{a_N} \frac{D_n}{a_N} \leq e^{-n}$ and \[ \sum_{|\alpha|=1}^N \sum_{0 \neq \beta \leq \alpha} \frac1{\beta! (\alpha-\beta)!} \frac1{(2R)^{\rho|\alpha|}} \leq \sum_{k=1}^N \Big( \frac{d}{(2R)^{\rho}}\Big)^k \] converges when $N\to+\infty$ provided $R\geq1$ be large enough.
\end{proof}

We want to prove that our class of pseudodifferential operators is closed when composing operators and also when we take transpose operators.

\begin{prop}\label{Prop231TesisDavid}
Let $P(x,D)$ be the pseudodifferential operator associated to $p(x,\xi) \in \GS^{m,\omega}_{\rho}$. Then the transpose operator, restricted to $\mathcal{S}_{\omega}$, can be decomposed as $P(x,D)^t = Q(x,D) + \widetilde R$, where $\widetilde R$ is an $\omega$-regularizing operator and $Q(x,D)$ is the operator defined by $q(x,\xi) \sim p^t(x,\xi)$.
\end{prop}
\begin{proof}
The transpose operator $P(x,D)^{t}$ is the pseudodifferential operator associated to the amplitude $p(y,-\xi)$. So, the result follows from Theorem~\ref{TheoLarguisimo}.
\end{proof}

The following result is straightforward, so we omit its proof~\cite{Z}.
\begin{lema}\label{Lemma232TesisDavid}
Let $p(x,\xi), q(x,\xi)$ be symbols in $\GS^{m,\omega}_{\rho}$. If $b(x,\xi)$ is a symbol in $\GS^{m,\omega}_{\rho}$ such that $b(x,\xi) \sim q^{t}(x,-\xi)$  and $r(x,\xi) \in \GS^{2m,\omega}_{\rho}$ is equivalent to $ \sum_j \sum_{|\alpha|=j} \frac1{\alpha!} \partial^{\alpha}_{\xi} D^{\alpha}_y (p(x,\xi)b(y,\xi))\left.\right|_{y=x}$, then $r(x,\xi) \sim p(x,\xi) \circ q(x,\xi)$.
\end{lema}

\begin{theo}\label{Theo233TesisDavid}
Let $p(x,\xi), q(x,\xi)$ be symbols in $\GS^{m_1,\omega}_{\rho}$, $\GS^{m_2,\omega}_{\rho}$ respectively, and let $P,Q:\mathcal{S}_{\omega} \to \mathcal{S}_{\omega}$ be the corresponding pseudodifferential operators. Then, the composition $P \circ Q: \mathcal{S}_{\omega} \to \mathcal{S}_{\omega}$  coincides, modulo an $\omega$-regularizing operator, with the pseudodifferential operator associated to  $(2\pi)^d (p(x,\xi) \circ q(x,\xi))$.
\end{theo}
\begin{proof}
We already know that $Q^t$ is given by the amplitude $q(y,-\xi)$. Then, $Q^t = Q' + T'$, where $T'$ is $\omega$-regularizing, and $Q'$ is defined by a symbol $q'$ that is equivalent to $q^t$. Since the class of the $\omega$-regularizing operators is closed by taking transposes, and by the fact that $(Q^t)^t=Q$, we observe $Q = Q_1 + T_1$, where $T_1$ is $\omega$-regularizing, and $Q_1$ is the operator associated to $b(y,\xi):=q'(y,-\xi) \sim q^t(y,-\xi)$. Moreover,  $P \circ T_1$ is an $\omega$-regularizing operator.

We consider the composition $P \circ Q_1: \mathcal{S}_{\omega} \to \mathcal{S}_{\omega}$ given by $P(Q_1 f)(x) = \int p(x,\xi) \widehat{Q_1f}(\xi) e^{ix\xi} d\xi.$ 
It is easy to see that  $Q_1f(x) = \widehat{I}(-x)$, where $I(\xi):= \int b(y,\xi) f(y) e^{-iy\xi} dy$. Thus, $\widehat{Q_1f}(\xi) = (2\pi)^d I(\xi)$, and hence $P\circ Q_1$ is a pseudodifferential operator associated to $(2\pi)^d p(x,\xi) b(y,\xi)$. Theorem~\ref{TheoLarguisimo} and  Lemma~\ref{Lemma232TesisDavid} give the conclusion.
\end{proof}

\textbf{Acknowledgements.} The first author was partially supported by the project GV Prometeo 2017/102, and the second author by the project MTM2016-76647-P. This article is part of the PhD. Thesis of V.~Asensio. The authors are very grateful to the two referees for the careful reading and  their suggestions and comments, which improved the paper.

\end{document}